\documentclass[12pt,twoside,a4paper]{article}
\usepackage{amsfonts}
\usepackage{graphicx}
\usepackage{algorithmic}
\usepackage{color}
\usepackage{amsmath}
\usepackage{subcaption}
\usepackage{amsopn}
\usepackage{amsthm}
\usepackage{amssymb}
\usepackage{graphicx}
\usepackage{algorithm}
\usepackage{dsfont}

\title{A Stochastic Gradient Method with Mesh Refinement for PDE Constrained Optimization under Uncertainty}
\author{Caroline Geiersbach\thanks{Department of Statistics and Operations Research, University of Vienna, A-1030 Vienna, Austria 
  (\texttt{caroline.geiersbach@univie.ac.at}).}
\and Winnifried Wollner\thanks{Fachbereich Mathematik, Technische Universit\"at Darmstadt, 64293 Darmstadt, Germany, 
  (\texttt{wollner@mathematik.tu-darmstadt.de})}}
\newcommand{\R}{\mathbb{R}}

\newcommand{\N}{\mathbb{N}}
\newcommand{\E}{\mathbb{E}}
\newcommand{\pP}{\mathbb{P}}
\newcommand{\D}{\mathop{}\!\mathrm{d}}

\newcommand{\U}{\mathcal{U}}
\newcommand{\tu}{\tilde{u}}
\newcommand{\Uad}{\mathcal{U}^{\text{ad}}}
\newcommand{\Dad}{D_{\text{ad}}}
\DeclareMathOperator*{\argmin}{arg\,min}
\DeclareMathOperator*{\esssup}{ess\,sup}
\numberwithin{equation}{section}

\newtheorem{thm}{thm}[section]
\newtheorem{lemma}[thm]{Lemma}
\newtheorem{definition}[thm]{Definition}

\newtheorem{assumption}[thm]{Assumption}

\theoremstyle{definition}
\newtheorem{remark}[thm]{Remark}

\begin{document}
\maketitle
\begin{abstract}
Models incorporating uncertain inputs, such as random forces or material parameters, have been of increasing interest in PDE-constrained optimization. In this paper, we focus on the efficient numerical minimization of a convex and smooth tracking-type functional subject to a linear partial differential equation with random coefficients and box constraints. The approach we take is based on stochastic approximation where, in place of a true gradient, a stochastic gradient is chosen using one sample from a known probability distribution. Feasibility is maintained by performing a projection at each iteration. In the application of this method to PDE-constrained optimization under uncertainty, new challenges arise. We observe the discretization error made by approximating the stochastic gradient using finite elements. Analyzing the interplay between PDE discretization and stochastic error, we develop a mesh refinement strategy coupled with decreasing step sizes. Additionally, we develop a mesh refinement strategy for the modified algorithm using iterate averaging and larger step sizes. The effectiveness of the approach is demonstrated numerically for different random field choices.
\end{abstract}

\section{Introduction}
In this paper, we are concerned with the numerical solution of a convex
optimization problem with convex constraints and an elliptic partial
differential equation (PDE) subject to uncertainty.  In applications,
the material coefficients and external inputs might not be known
exactly. They can then be modeled to be distributed according to a known probability distribution. When the number of possible
scenarios in the probability space is small, then the optimization
problem can be solved over the entire set of scenarios. This approach
is not relevant for most applications, as it becomes intractable if
the source of uncertainty contains more than a few scenarios. Solvers for problems with random PDEs generally use either a discretization of the stochastic space or rely on sampling. Methods with a discretized stochastic space include the stochastic Galerkin method \cite{Babuska2004} and sparse-tensor discretization \cite{Schwab2011}. Sample-based approaches involve taking random or carefully chosen realizations of the input parameters; these approaches include Monte Carlo or quasi Monte Carlo methods and stochastic collocation \cite{Babuska2007}. 

In PDE-constrained optimization under uncertainty, there are several main algorithmic approaches. Most approaches involve using deterministic optimization methods in combination with a sampling or discretization scheme for the stochastic space. Stochastic collocation has been combined with multi-grid methods \cite{Borzi2009}, gradient descent and SQP methods \cite{Tiesler2012}, and a trust region method \cite{Kouri2013}. In combination with sparse-grid collocation or low-rank tensors, trust-region methods have been proposed \cite{Kouri2014, Garreis2019+}.  Discretization of both spatial and stochastic spaces have been proposed in \cite{Hou2011}, and with a one-shot approach with stochastic Galerkin finite elements in \cite{Rosseel2012}. All of these methods suffer from the curse of dimensionality -- as the stochastic dimension increases, the number of quadrature points must increase exponentially. 

Sample average approximation, also known as the Monte Carlo method, involves replacing the stochastic integral with a fixed sample of randomly chosen points. While the error in a Monte Carlo estimator decreases as $\mathcal{O}(1/\sqrt{N})$, where $N$ is the number of sampled points, this rate is independent of the stochastic dimension. There are known improvements to substantially improve the slow convergence of this method that have been developed for this problem class, including the multilevel Monte Carlo method \cite{Ali2017} or the quasi Monte Carlo method \cite{Guth2019+}. In the context of approaches independent of the stochastic dimension, it is also worth mentioning the work of \cite{Alexanderian2017} and a following work \cite{Chen2019}, which is fundamentally different from the above approaches; this approach relies on Taylor expansions with respect to the parameter of the parameter-to-objective map.

Recently, stochastic approximation methods have been investigated for
efficiently solving PDE-constrained optimization problems involving
uncertainty \cite{Haber2012, Martin2018, Geiersbach2018}. This
approach has previously been unexploited for PDE-constrained
optimization, even though it is a classical method for solving
stochastic optimization problems dating back to the 1950s
\cite{Robbins1951, Kiefer1952}. The main tool in stochastic
approximation is a stochastic gradient, in place of the true gradient,
to iteratively minimize the expected value over a random
function. This method differs from the approaches mentioned above
  in that it is a fundamentally random method; sampling is performed
  in the course of the optimization procedure, rather than in addition to
  it. Like sample average approximation, it enjoys convergence rates independent of the stochastic dimension.
In \cite{Haber2012}, the authors compare the stochastic
approximation approach with the sample average approximation method
for a fully discrete (both spatially and stochastically)
PDE-constrained optimization problem, but they do not handle
additional constraints or PDE discretization error. A mesh refinement
strategy was presented in \cite{Martin2018}, but only in combination
with step sizes of the form $c/n$; additionally, their
results do not handle the case with additional constraints or with
iterate averaging. Convergence theory with additional constraints in
Hilbert spaces was presented in \cite{Geiersbach2018} along with a
summary of step size rules, both for strongly convex and generally
convex objective functionals; however, PDE discretization error was
not handled in this work. In this work, we will extend the results in
\cite{Geiersbach2018} to incorporate bias by PDE discretization
error. We will see that we can obtain the same convergence theory, with
the same expected error decay, if the discretization accuracy is
steered such that the bias decays fast enough. 

Relying on a priori error estimate for the discretization error, we
provide a rule how the maximal mesh size should be coupled with the
iteration progress. Analogously, one could couple the iteration with
some a posteriori error measure, which has been well investigated for
deterministic PDE-constrained optimization problems, see,
e.g.,~\cite{Rannacher2010, RannacherVexlerWollner:2011}, and including the treatment of inexact discrete solutions~\cite{MeyerRademacherWollner:2015}.

The paper is structured as follows. In section~\ref{section:preliminaries}, the algorithm and notation is presented. In section~\ref{sec:PSG-Efficiency}, efficiency estimates are derived for different step sizes choices. An application to PDE-constrained optimization is introduced in section~\ref{sec:ModelProblem}, and a discretized version of the algorithm is presented. The presented version allows the coupling of step size rules to successive mesh refinement. Convergence orders for the algorithm are presented in Theorem~\ref{thm:main-result}, which is our main result. Experiments supporting the theoretical work are in section~\ref{sec:Experiments}, and we close with final remarks in section~\ref{section:conclusion}.

\section{Preliminaries}
\label{section:preliminaries}
We consider problems of the form
\begin{equation}\label{eq:SAproblem}
 \min_{u \in \Uad} \left\lbrace j(u) = \E[J(u,\xi)] = \int_\Omega J(u,\xi(\omega)) \D\pP (\omega)\right\rbrace,
\end{equation}
where $\Uad$ is a nonempty, closed, and convex subset of a Hilbert
space $(\U, (\cdot,\cdot)_{\U})$. We recall that a probability space is given by a triple $(\Omega, \mathcal{F}, \pP)$, where $\Omega$ represents the sample space, $\mathcal{F} \subset 2^{\Omega}$ is the $\sigma$-algebra of events and $\pP\colon \Omega \rightarrow [0,1]$ is a probability measure defined on $\Omega$. For the random vector $\xi: \Omega \rightarrow \Xi \subset \R^m$, we will often denote a realization of the random vector as simply $\xi \in \Xi.$ It is assumed that for almost every $\omega$, $u \mapsto J(u,\xi(\omega))$ is convex on $\Uad$, making $j$ convex as well.
Additionally, we require that $J:\U \times \Xi \rightarrow \R$ is
$L^2$-Fr\'echet differentiable on an open neighborhood of $\Uad$
according to the following definition, where $L^p(\Omega)$ denotes the
space of all $p$-times integrable real-valued functions with norm
$\lVert f \rVert_{L^p(\Omega)} = (\int_\Omega |f(\omega)|^p \D
\pP(\omega))^{1/p}$ and $\lVert \cdot \rVert_{\U} = \sqrt{( \cdot,
  \cdot )_{\U}}$ denotes the (strictly convex) norm on $\U$.

For the convenience of the reader, we recall the following definition from \cite{Geiersbach2018}.
\begin{definition}
\label{definition-Lp-differentiable}
A $p$-times integrable random functional $J:\U \times \Xi \rightarrow \R$ is called $L^p$-Fr\'echet differentiable at $u$ if for an open set $U \subset \U$ containing $u$ there exists a bounded and linear random operator $A:U \times \Xi \rightarrow \R$ such that $\lim_{h \rightarrow 0} \lVert J(u + h, \xi) - J(u, \xi) + A(u,\xi)h \rVert_{L^p(\Omega)} / \lVert h \rVert_\U = 0$.
\end{definition}
By H\"older's inequality, if $u \mapsto J(u,\cdot)$ is
$L^p$-differentiable and $1 \leq r < p$, then it is also
$L^r$-differentiable with the same derivative. This implies that $j:\U
\rightarrow \R$ is Fr\'echet
differentiable.\footnote{Definition~\ref{definition-Lp-differentiable} with $p=1$ is the
  minimal requirement for allowing the exchange of the derivative and
  the expectation, i.e.,~$\nabla j(u) = \int_{\Omega} \nabla_u
  J(u,\xi(\omega)) \D \pP(\omega)$. A sufficient, but not necessary, condition for this is that (i) $j(v)$ is finite for all $v \in U$ and $u \mapsto J(u,\xi)$ is a.s.~Fr\'{e}chet differentiable at $u$; and (ii) there exists an $\pP$-integrable dominating function $g$ such that for all $v \in U$, $\lVert \nabla_u J(v,\xi)\rVert_{\U} \leq g(\xi)$ a.s.}

The projection onto a closed convex set $\Uad \subset \U$ is denoted by $\pi_{\Uad}:\U \rightarrow \Uad$ and is defined as the function such that
$$ \pi_{\Uad}(u) = \underset{w \in \Uad}{\argmin} \,  \lVert u-w \rVert_{\U}.$$
The projected stochastic gradient (PSG) method, which is studied in this paper, is summarized in Algorithm \ref{alg:PSG_Hilbert}. It relies on a stochastic gradient, or a function $G:\U \times \Xi \rightarrow \U$ such that $G(u,\xi) \approx \nabla \E[J(u,\xi)]$; one choice for $G(u,\xi)$ is $\nabla_u J(u,\xi)$.
\begin{algorithm}
\begin{algorithmic}[1] 
\STATE \textbf{Initialization:} $u^1 \in \U$
\FOR{$n=1,2,\dots$}
\STATE Generate $\xi^n$, independent from $\xi^1, \dots, \xi^{n-1}$, and $t_n >0$.
\STATE $u^{n+1} := \pi_{\Uad}(u^n - {t_n}G(u^n,\xi^n))$.
\ENDFOR
\end{algorithmic}
\caption{Projected Stochastic Gradient (PSG) Method}
\label{alg:PSG_Hilbert}
\end{algorithm}

We recall that a sequence $\{ \mathcal{F}_n\}$ of increasing sub-$\sigma$-algebras of $\mathcal{F}$ is called a filtration. A stochastic process $\{\beta_n\}$ is said to be adapted to the filtration if $\beta_n$ is $\mathcal{F}_n$-measurable for all $n$. If 
$$\mathcal{F}_n = \sigma(\beta_1, \dots, \beta_n),\footnote{The $\sigma$-algebra generated by a random variable $\beta:\Omega \rightarrow \R$ is given by $\sigma(\beta) = \{ \beta^{-1}(B): B \in \mathcal{B}\}$, where $\mathcal{B}$ is the Borel $\sigma$-algebra on $\R$. Analogously, the $\sigma$-algebra generated by the set of random variables $\{ \beta_1, \dots, \beta_n\}$ is the smallest $\sigma$-algebra such that $\beta_i$ is measurable for all $i=1, \dots, n$.}$$
we call $\{ \mathcal{F}_n\}$ the natural filtration. Furthermore, we define for an integrable random variable $\beta:\Omega \rightarrow \R$ the conditional expectation $\E[\beta | \mathcal{F}_n]$, which is itself a random variable that is $\mathcal{F}_n$-measurable and satisfies $\int_A \E[\beta(\omega) | \mathcal{F}_n] \D \pP(\omega) = \int_A \beta(\omega) \D \pP(\omega)$ for all $A \in \mathcal{F}_n$.

We make the similar assumptions on the  gradient as \cite{Geiersbach2018}; for the purposes of this paper, we will focus on the case where $\Uad$ is bounded.

\begin{assumption}\label{assumption:gradient}
Let $\{\mathcal{F}_n \}$ be an increasing sequence of
$\sigma$-algebras and the  sequence of stochastic gradients generated by Algorithm~\ref{alg:PSG_Hilbert} be given by $\{G(u^n,\xi^n)\}$.  For each $n$, there exist $r^n$, $w^n$ with
$$r^n = \E[G(u^n,\xi^n) | \mathcal{F}_n] - \nabla j(u^n), \quad w^n =
G(u^n, \xi^n) - \E[G(u^n, \xi^n) | \mathcal{F}_n],$$
which satisfy the following assumptions: (i) $ u^n$ and $r^n $ are $\mathcal{F}_n$-measurable; (ii) $K_n:=\esssup_{\omega \in \Omega} \lVert r^n(\omega)\rVert_{\U}$ is bounded, i.e.,~\mbox{$\sup_n K_n < \infty$}; (iii) there exists a constant $M > 0$ such that $\E[\lVert G(u,\xi)\rVert_{\U}^2 ] \leq M$ for all $u \in \Uad.$
\end{assumption}

Notice that by construction, $\E[w^n | \mathcal{F}_n] = 0$ and hence no
further assumptions on $w^n$ are needed.

\section{Efficiency Estimates for Stochastic Gradient
  Methods}\label{sec:PSG-Efficiency}
To obtain efficiency estimates, we let $u$ be an optimal solution of \eqref{eq:SAproblem} and $g^n = G(u^n,\xi^n)$. Since $u \in \Uad$, $\pi_{\Uad}(u) = u.$ Thus, the nonexpansivity of the projection operator yields
 \begin{equation}\label{eq:boundsequence}
 \begin{aligned}
\lVert u^{n+1} - u \rVert_{\U}^2 &= \lVert \pi_{\Uad}(u^n - t_n g^n) - \pi_{\Uad}(u) \rVert_{\U}^2 \\
&\leq  \lVert u^n - t_n g^n - u \rVert_{\U}^2\\
& = \lVert u^n - u \rVert_{\U}^2 - 2 t_n ( u^n - u, g^n )_{\U} + t_n^2 \lVert g^n \rVert_{\U}^2.
\end{aligned}
 \end{equation}
Since $\xi^n$ is independent from $\xi^1, \dots, \xi^{n-1}$, it follows that
\begin{equation}
\label{eq:gradient-martingale-to-mean}
 \E[\lVert g^n \rVert_{\U}^2 | \mathcal{F}_n] = \E[\lVert G(u^n,\xi)\rVert_{\U}^2] \leq M.
\end{equation}
By Assumption~\ref{assumption:gradient}, $g^n = \nabla j(u^n) + w^n + r^n.$ Since $u^n$ and $ r^n$ are $\mathcal{F}_n$-measurable, it follows that \mbox{$\E[u^n |\mathcal{F}_n ] = u^n$} and $\E[r^n |\mathcal{F}_n ] = r^n$. Note as well that $\E[w^n|\mathcal{F}_n] = 0$ holds. Thus taking conditional expectation with respect to $\mathcal{F}_n$ on both sides of \eqref{eq:boundsequence}, we get
\begin{equation}\label{eq:fundamental-inequality-efficiency}
\begin{aligned}
 &\E[\lVert u^{n+1} - u \rVert_{\U}^2 | \mathcal{F}_n ]  \leq \lVert u^n - u \rVert_{\U}^2 - 2 t_n ( u^n - u, \nabla j(u^n) + r^n )_{\U} + t_n^2 M.
\end{aligned}
\end{equation}
In the following computations, let $e_n^2:=\E[\lVert u^n - u \rVert_{\U}^2]$.

\subsection{Strongly Convex Case}
Notice that
   $$- 2t_n ( u^n - u, r^n )_{\U} \leq 2 t_n (\lVert u^n - u \rVert_{\U}^2 + 1) \lVert r^n \rVert_{\U}$$ 
and the $\mu$-strong convexity of $j$ implies that $( u^n - u, \nabla j(u^n) )_{\U} \geq \mu \lVert u^n - u\rVert_{\U}^2$. Hence, taking expectation on both sides of \eqref{eq:fundamental-inequality-efficiency}, we obtain
 \begin{equation*}
  e_{n+1}^2 \leq e_n^2 (1 - 2 \mu t_n + 2 t_n K_n ) + t_n^2 M + 2 t_n K_n. 
 \end{equation*}
 To ensure convergence of $\{e_n^2 \}$, we require that $\sum_n t_n K_n < \infty$ and $\sum_n t_n^2 < \infty$; see \cite[Theorem 3.6]{Geiersbach2018}. We use for some later to be determined $K,\nu,\theta > 0$ the ansatz
\begin{equation}\label{eq:Kstrongly-convex}
  K_n \leq \frac{K}{n + \nu}, \quad t_n = \frac{\theta}{n+\nu},
\end{equation}
resulting in the inequality
\begin{equation}
 \label{eq:recursion-strongly-convex}
e_{n+1}^2 \leq e_n^2 \left(1 - \frac{ 2 \mu \theta}{n +\nu}  + \frac{2\theta K}{(n+\nu)^2} \right) + \frac{\theta^2 M +2 \theta K}{(n+\nu)^2}.
\end{equation}

\begin{lemma}\label{lemma:recursion-stronglyconvex}
For a recursion of the form
\begin{equation}
\label{eq:recursion-relation}
 e_{n+1}^2 \leq e_n^2 \left(1 - \frac{c_1}{n+\nu} +\frac{c_2}{(n+\nu)^2}\right)+\frac{c_3}{(n+\nu)^2},
 \end{equation}
if $e_1^2, c_2, c_3 \geq 0$, $c_1 > 1$, and $\nu+1 \geq \frac{c_2}{c_1-1}$,
it follows that
\begin{equation}\label{eq:recursion-formula}
e_n^2 \leq \frac{\rho}{n+\nu},
\end{equation}
where 
$$\rho:=\max \Big\lbrace (1+\nu)e_1^2,  \frac{-c_3(1+\nu)}{(1+\nu)(1-c_1)+c_2}\Big\rbrace.$$ 
\end{lemma}
\begin{proof}
We show \eqref{eq:recursion-formula} by induction. The statement for
$n=1$ is clearly satisfied since $e_1^2 = \frac{\nu+1}{\nu+1} e_1^2 \leq \frac{\rho}{\nu+1}.$

For $n>1$, we assume that \eqref{eq:recursion-formula} holds for
$n$. We abbreviate $\hat{n}:=n+\nu$ and since $\nu+1 \ge
\frac{c_2}{c_1-1}$, we have
\[
  1 - \frac{c_1}{\hat{n}} +\frac{c_2}{\hat{n}^2} > 0.
\]
Thus by \eqref{eq:recursion-relation} and
\eqref{eq:recursion-formula}, we get 
\begin{equation}\label{eq:lem:recursion-formula}
  \begin{aligned}
    e_{n+1}^2 &\leq  \left(1 - \frac{c_1}{\hat{n}} +\frac{c_2}{\hat{n}^2}\right) \frac{\rho}{\hat{n}}+\frac{c_3}{\hat{n}^2}\\
    & = \left( \frac{\hat{n}^2-\hat{n}}{\hat{n}^3}\right)\rho+ \left(\frac{\hat{n}(1-c_1)+c_2}{\hat{n}^3} \right)\rho+\frac{c_3}{\hat{n}^2}\\
    & \leq \frac{\rho}{\hat{n}+1}.
  \end{aligned}
\end{equation}
In the last inequality, we used the fact that $\hat{n}^3 \geq
\hat{n}(\hat{n}-1)(\hat{n}+1)$ and the fact that for all $n \in \N$ and
$\hat{n} = n+\nu$, the relation
$$\left(\frac{\hat{n}(1-c_1)+c_2}{\hat{n}^3}
\right)\rho+\frac{c_3}{\hat{n}^2}\leq
\left(\frac{\hat{n}(1-c_1)+c_2}{\hat{n}^3}
\right)\frac{-c_3(1+\nu)}{(\nu+1)(1-c_1)+c_2} +\frac{c_3}{\hat{n}^2}$$
is true,
since the factor in front of $\rho$ is negative by assumption on
$\nu$, i.e., $(\nu+1)(1-c_1) + c_2 \leq 0$.
Further, we calculate 
\begin{align*}
  \left(\frac{\hat{n}(1-c_1)+c_2}{\hat{n}^3} \right)\frac{-c_3(1+\nu)}{(1+\nu)(1-c_1)+c_2} +\frac{c_3}{\hat{n}^2} &\leq 0 \\
\Leftrightarrow -c_3(1+\nu)[(n+\nu)(1-c_1)+c_2] &\geq -c_3(n+\nu)[(1+\nu)(1-c_1)+c_2]\\
\Leftrightarrow (1+\nu) &\leq (n+\nu), 
\end{align*}
thus showing~\eqref{eq:lem:recursion-formula}.
\end{proof}

Summarizing the above derivation, we obtain the following
convergence theorem.
\begin{thm}
\label{thm:strongly-convex-step-size-bias-rule}
If $j$ is $\mu$-strongly convex and $\theta$ and $\nu$ are chosen such that $\theta > 1/(2\mu)$ and $\nu \geq 2\theta K/(2\mu \theta -1) -1$, then 
\begin{equation}
\label{eq:efficiency-iterates-strongly-convex}
 \E[\lVert u^n - u \rVert_{\U}] \leq \sqrt{\frac{\rho}{n+\nu}}
\end{equation}
with
$$\rho:= \max \left\lbrace (1+\nu) \E[\lVert u^1-u \rVert_{\U}^2], \frac{-(\theta^2 M+ 2\theta K)(1+\nu)}{(1+\nu)(1-2\mu \theta) + 2\theta K}\right\rbrace.$$

If additionally, $\nabla j$ is Lipschitz continuous with constant $L>0$ and $\nabla j(u) = 0$, then
\begin{equation}\label{eq:objectiveerrorbounds}
\E[j(u^n)-j(u)] \leq \frac{L\rho}{2(n+\nu)}. 
\end{equation}
\end{thm}

\begin{proof}
 The estimate \eqref{eq:efficiency-iterates-strongly-convex} is an immediate consequence of \eqref{eq:recursion-strongly-convex} and Lemma~\ref{lemma:recursion-stronglyconvex}. If $\nabla j$ is Lipschitz continuous and $\nabla j(u) = 0$, then it follows that
\begin{equation}\label{eq:strongly-convex-interior-point}
j(u^n) \leq j(u) + \frac{L}{2} \lVert u^n - u \rVert_{\U}^2,
\end{equation}
so combining~\eqref{eq:strongly-convex-interior-point}
with~\eqref{eq:efficiency-iterates-strongly-convex}, we get \eqref{eq:objectiveerrorbounds}.
\end{proof}

\subsection{Convex Case with Averaging}
In the general convex case, or where a good estimate for $\mu$ does
not exist, step sizes of the form $t_n = \theta/n$ may be too small
for efficient convergence. An example is given in \cite{Nemirovski2009} showing that an overestimated strong convexity parameter $\mu$ leads to extremely slow convergence. A significant improvement can be obtained by using larger steps of the order $\mathcal{O}(1/\sqrt{n})$. Then, instead of observing convergence of the sequence $\{ u^n\}$ we observe the convergence of certain averages $\tilde u_i^N$ of the iterates, with
$\gamma_n:=t_n/(\sum_{\ell=i}^N t_\ell)$ and the average of the iterates for some choice of $i$ to $N$ given by
\begin{equation}\label{eq:averagediterates}
\tilde{u}_i^N = \sum_{n=i}^N \gamma_n u^n.
\end{equation}
To derive these estimates, we use \eqref{eq:fundamental-inequality-efficiency} and the fact that
$( u^n - u, \nabla j(u^n) )_{\U} \geq j(u^n) - j(u)$ by convexity of $j$ to get a recursion of the form
\begin{equation}
\label{eq:recursion-convex}
\begin{aligned}
e_{n+1}^2 &\leq e_n^2 (1+2 t_n K_n) - 2t_n \E[j(u^n)-j(u)] + t_n^2 M +2 t_n K_n.
\end{aligned}
 \end{equation}  
Rearranging \eqref{eq:recursion-convex} and summing over $1 \leq i
\leq N$ on both sides,
\begin{equation}
\begin{aligned}
\label{eq:estimate-jerror-convex}
\sum_{n=i}^N t_n \E[j(u^n)-j(u)] &\leq \sum_{n=i}^N \left[ \frac{e_n^2}{2} (1+ 2 t_n K_n) - \frac{e_{n+1}^2}{2} + \frac{t_n^2 M}{2} +  t_n K_n\right]\\ 
&\leq \frac{e_i^2}{2} + \frac{1}{2} \sum_{n=i}^N  \left[ 2 t_n K_n e_n^2 + t_n^2 M +  2 t_n K_n\right].
\end{aligned}
\end{equation}
By convexity of $j$, we have $j(\tu_i^N) \leq \sum_{n=i}^N \gamma_n j(u^n)$ so by \eqref{eq:estimate-jerror-convex}
\begin{equation}
\label{eq:basic-efficiency-estimate-convex}
 \E[j(\tu_i^N)-j(u)] \leq \frac{e_i^2 + \sum_{n=i}^N  [2 t_n K_n e_n^2 + t_n^2 M +  2 t_n K_n] }{2 \sum_{n=i}^N t_n}.
\end{equation}
Set $\Dad: = \sup_{u \in \Uad}  \lVert u_1 - u \rVert_{\U}$. Notice that $e_1^2
\leq \Dad^2$ and $e_i^2 \leq 4 \Dad^2$ since $\lVert u_i - u\rVert_{\U} \leq \lVert u_i - u_1\rVert_{\U} + \lVert u_1 - u \rVert_{\U} \leq 2\Dad$. 
 Thus from \eqref{eq:basic-efficiency-estimate-convex} we get 
\begin{align}
 \E[j(\tu_1^N)-j(u)] &\leq \frac{\Dad^2 + \sum_{n=1}^N \left[ 8 t_n K_n \Dad^2 + t_n^2 M +  2 t_n K_n \right] }{2 \sum_{n=1}^N t_n},\label{eq:estimate-jerror-convex-i1}\\
 \E[j(\tu_i^N)-j(u)] &\leq \frac{4 \Dad^2 + \sum_{n=i}^N \left[ 8 t_n K_n \Dad^2 + t_n^2 M +  2 t_n K_n\right] }{2 \sum_{n=i}^N t_n}, \quad 1 < i \leq N \label{eq:estimate-jerror-convex-i2}.
\end{align}
If $K_n = 0$, then we recover the estimates~\cite[(2.18)]{Nemirovski2009}.

\paragraph{Constant Step Size Policy}
First, observe the case where $t_n = t$ and $i=1$. It follows by \eqref{eq:estimate-jerror-convex-i1} that
$$ \E[j(\tu_1^N)-j(u)] \leq \frac{\Dad^2 + \sum_{n=1}^N  \left[8 t K_n \Dad^2 + t^2 M +  2 t K_n\right] }{2 N t}$$
Minimizing $f(t):= (\Dad^2 + \sum_{n=1}^N \left[ 8 t K_n \Dad^2 + t^2 M +  2 t K_n\right]) /(2 N t)$, we get the step size policy
\begin{equation}\label{eq:stepsizerule-convex-constant}
 t = \frac{\Dad}{\sqrt{M N}},
\end{equation}
which is the same step size rule as one would use where $K_n = 0.$ Plugging \eqref{eq:stepsizerule-convex-constant} into \eqref{eq:estimate-jerror-convex-i1}, we get 
$$\E[j(\tu_1^N)-j(u)] \leq \frac{\Dad \sqrt{M}}{\sqrt{N}} + \frac{4 \Dad^2+1}{N} \sum_{n=1}^N K_n.$$
Hence for convergence with the same speed as in the case $K_n = 0$ it
is sufficient that 
\begin{equation}\label{eq:Kconvex-constant}
  \sum_{n=1}^N K_n \propto \sqrt{N}.
\end{equation}

\paragraph{Variable Step Size Policy}
Alternatively, one can work with the decreasing step size policy for a constant $\theta > 0$
\begin{equation}\label{eq:stepsizerule-convex}
 t_n = \frac{\theta \Dad}{\sqrt{M n}}.
\end{equation}
Plugging \eqref{eq:stepsizerule-convex} into \eqref{eq:estimate-jerror-convex-i2}, we get using the inequalities $$\sum_{n=i}^N \frac{1}{n} \leq \frac{N-i+1}{i}, \quad \sum_{n=i}^N \frac{1}{\sqrt{n}} \geq \frac{N-i+1}{\sqrt{N}}$$
the following estimate for $1 \leq i \leq N$ 
\begin{equation*}
\quad \E[j(\tilde{u}_i^N) - j(u)] \leq \frac{1}{\sqrt{N}} \left[ \frac{2 \Dad N \sqrt{M}}{\theta (N-i+1)} +\frac{(4\Dad^2+1)N}{N-i+1} \sum_{n=i}^N \frac{K_n}{\sqrt{n}} + \frac{\theta \Dad \sqrt{M} N}{2 i} \right].
\end{equation*}
Hence to balance the terms it is suitable to select 
\begin{equation}\label{eq:Kconvex-variable}
  \sum_{n=i}^N \frac{K_n}{\sqrt{n}} \propto 1
\end{equation}
and $i = \lceil\alpha N\rceil$ for some $\alpha \in (0,1)$.

We summarize the convergence rate for iterate averaging in the general convex case in the following theorem.
\begin{thm}
\label{thm:convex-averaging}
If $j$ is convex and iterates are averaged according to \eqref{eq:averagediterates}, then with the constant step size policy \eqref{eq:stepsizerule-convex-constant} and bias $K_n$ satisfying \eqref{eq:Kconvex-constant}, we have
$$\E[j(\tu_1^N)-j(u)] \leq \mathcal{O}\left(\frac{1}{\sqrt{N}}\right).$$
If variable step sizes are chosen according to \eqref{eq:stepsizerule-convex} and bias satisfies \eqref{eq:Kconvex-variable} for $i = \lceil\alpha N\rceil$ and some $\alpha \in (0,1)$, it follows
$$\E[j(\tilde{u}_i^N) - j(u)] \leq \mathcal{O}\left(\frac{1}{\sqrt{N}}\right).$$
\end{thm}

\section{Application to PDE-Constrained Optimization under Uncertainty}\label{sec:ModelProblem}
Let $D \subset \R^2$ be a convex polygonal domain. 
We set $\U = L^2(D)$ and $(\cdot,\cdot)_{\U} = (\cdot,\cdot)_{L^2(D)}$
and use the same notation also for vector-valued functions. Let $\mathcal{Y}^0:=H_0^1(D)$. Further, let $|\cdot|_{H^k(D)}$ and $\lVert \cdot \rVert_{H^k(D)}$ be the seminorm and norm on the Sobolev space $H^k(D)$, respectively; see \cite{Adams2003} for a definition of these norms.  We denote the set of $t$-H\"older continuous functions on $\bar{D}$ with $C^t(\bar{D})$. For $1 \leq p < \infty$, a measure space $(\Xi, \mathcal{X}, P)$ and Banach space $(X,\lVert \cdot \rVert_X)$, the Bochner spaces $L^p(\Xi,X)$ and $L^\infty(\Xi,X)$ are defined as the sets of strongly $\mathcal{X}$-measurable functions $y:\Xi \rightarrow X$ such that
\begin{align}
 \lVert y \rVert_{L^p(\Xi,X)}&:=\left( \int_{\Xi}\lVert y(\xi) \rVert_X^p \D P(\xi) \right)^{1/p},\\
 \lVert y \rVert_{L^\infty(\Xi,X)}&:=\esssup_{\xi \in \Xi} \lVert y(\xi) \rVert_X
\end{align}
are finite, respectively.

Let $(\Omega, \mathcal{F}, \mathbb{P})$ be a probability space. We consider the constraint, to be satisfied $\mathbb{P}$-a.s., of the form 
\begin{equation}\label{eq:randomPDE}
 \begin{aligned}
  \quad -   \nabla \cdot  (a(x,\omega) \nabla y(x,\omega)) &= u(x), \qquad x \in D, \\
   y(x,\omega) &= 0, \phantom{tex}\qquad x \in \partial D,
 \end{aligned}
\end{equation}
where $a: D \times \Omega \rightarrow \R$ is a random field representing conductivity on the domain. To facilitate simulation, we will make a standard finite-dimensional noise assumption, meaning the random field has the form 
  $$a(x,\omega) = a(x,\xi(\omega)) \quad \text{ in } D \times \Omega$$
where $\xi(\omega) = (\xi_1(\omega), \dots, \xi_m(\omega))$ is a vector of real-valued uncorrelated random variables $\xi_{i}:\Omega \rightarrow \Xi_i \subset\R$.\footnote{We use $\xi_i$ to denote the $i^{\text{th}}$ element of the vector $\xi$ and $\xi^n$ to denote the $n^{\text{th}}$ realization of the vector $\xi^n=(\xi_1^n, \dots, \xi_m^n)$.} The support of the random vector will be denoted with $\Xi = \prod_{i=1}^m \Xi_i$ and its probability distribution with $P$. By assumption on $a$, it is possible to reparametrize $y$ as likewise depending on $\xi$, see \cite[Lemma 9.40]{Lord2014}. Therefore, we can associate the random field $y$ with a function $y = y(x,\xi)$ belonging to the space $L^2(\Xi,\mathcal{Y}^0).$ Now, the problem of finding a $u \in \Uad$ bounded by $u_a, u_b \in \U$ such that the corresponding $y \in L^2(\Xi, \mathcal{Y}^0)$ best approximates a target temperature $y^D \in L^2(D)$ with cost $\lambda \geq 0$ is formulated in \eqref{eq:stationaryheatsourceproblemrandom}.
\begin{equation}\label{eq:stationaryheatsourceproblemrandom}
 \begin{aligned}
   \min_{u \in \Uad} \quad \Bigg\lbrace j(u):= \E[ J(u, \xi) ] &:= \E \left[\frac{1}{2} \lVert y  - y^D\rVert_{\U}^2 \right]  + \frac{\lambda}{2} \lVert u \rVert_{\U}^2 \Bigg\rbrace
   \\
   \text{s.t.} \quad -   \nabla \cdot  (a(x,\xi) \nabla y(x,\xi)) &= u(x), \qquad (x,\xi) \in D \times \Xi, \\
   y(x,\xi) &= 0, \phantom{tex}\qquad (x,\xi) \in \partial D \times \Xi,\\
   \Uad:= \{ u \in \U:  \,u_a(x) & \leq u(x) \leq u_b(x)\,\,\text{ a.e. } x\in D\}.
 \end{aligned}
\end{equation}

We will often suppress dependence on $x$ and simply write $a(\xi) = a(\cdot,\xi)$ and $y(\xi) = y(\cdot,\xi)$ for a realization of the random field and temperature, respectively. The random field is subject to the following assumption.
\begin{assumption}
\label{assumption:regularity-random-field}
There exist $a_{\min}, a_{\max}$ such that for almost every $(x,\xi) \in D \times \Xi,$
$0 < a_{\min} < a(x,\xi) < a_{\max} < \infty$.
Additionally, $a \in L^\infty(\Xi,C^t(\bar{D}))$ for some $0<t\leq 1$.
\end{assumption}

\begin{remark}
Assumption~\ref{assumption:regularity-random-field} allows for modeling with log-normal random fields with truncated Gaussian noise, as in for instance \cite{Gunzburger2014} and \cite{Ullmann2012}.  The H\"older condition $a \in L^\infty(\Xi,C^t(\bar{D}))$ is weaker than the typical assumption, where the fields are assumed to be almost surely continuously differentiable with uniformly bounded gradient; see for instance \cite{Babuska2004} and \cite{Martin2018}.
\end{remark}

\begin{lemma}\label{lemma:poissonrandom}
Let Assumption~\ref{assumption:regularity-random-field} be satisfied for some $t \in
(0,1]$. Then there exists some $s_0 \in (0,t]$ such that 
for any $0\le s < s_0$, any $u \in H^{s_0-1}(D)$, and almost every $\xi \in \Xi$ there exists a unique solution
$y(\xi) \in \mathcal{Y}^0\cap H^{1+s}(D)$ to
\begin{equation}
\label{eq:Poissonrandom}
b^\xi(y,v):=\int_D a(x,\xi) \nabla y(x,\xi) \cdot \nabla v(x) \D x = \int_D u(x) v(x) \D x =: (u,v)_{\U}
\end{equation}
for all  $v \in \mathcal{Y}^0$. Moreover, for any such $s$ there exists
  $C_s$ independent of $\xi$ and $u$ such that
\begin{equation} \label{eq:Aprioribounds_Poisson_stochastic}
 \lVert y(\xi)\rVert_{H^{1+s}(D)} \leq C_s \lVert u
 \rVert_{H^{s-1}(D)}.
\end{equation}
Additionally, if $D$ is convex and $t = 1$, then the statement remains
true for $s = s_0 = 1$.
\end{lemma}
\begin{proof}
  It is an immediate consequence of the Lax--Milgram Lemma and the
  bounds on $a(\xi)$ from Assumption~\ref{assumption:regularity-random-field}
  that~\eqref{eq:Poissonrandom} has a unique solution in $\mathcal{Y}^0$
  and~\eqref{eq:Aprioribounds_Poisson_stochastic} holds with $s=0$.
  The existence of $s_0$ and the regularity in $H^{1+s}$ follows  
  from~\cite[Lemma~1 and
  Theorem~1]{HallerDintelmannMeinlschmidtWollner:2018}.

  In the case of a convex
  domain and $t=1$,~\cite[Theorem~3.2.1.2]{Grisvard:1985}
  provides the regularity $y(\xi) \in H^2(D)$ for the solution
  of~\eqref{eq:Poissonrandom}. The a priori
  bound~\eqref{eq:Aprioribounds_Poisson_stochastic} follows
  from~\cite[Theorem~3.1.3.1]{Grisvard:1985} and inspection of the 
  proof of~\cite[Theorem~3.2.1.2]{Grisvard:1985}, showing that the
  bound also remains true for an arbitrary convex domain.
\end{proof}

Note that similar estimates, even with $s_0 = t$, can be shown for smooth
domains, see, e.g.,~\cite[Proposition 3.1]{Charrier2013}. 

Using standard arguments, it can be shown that for $\xi \in \Xi$, the stochastic gradient $\nabla_u J(u,\xi)$ for problem \eqref{eq:stationaryheatsourceproblemrandom} is given by
\begin{equation}\label{eq:stochastic-gradient}
\nabla_u J(u,\xi) = \lambda u - p(\cdot,\xi),
\end{equation}
where $p(\cdot,\xi) \in \mathcal{Y}^0$ solves $b^\xi(v,p) = (y^D-y(\xi),v)_{\U}$ for all $v \in \mathcal{Y}^0$; see \cite{Geiersbach2018}.

\subsection{Discretization}
\label{subsection:discretization}
We now define a discretization
of~\eqref{eq:stationaryheatsourceproblemrandom} by finite elements.
To this end, let $\mathcal T_h$ be a decomposition of $D$ into
shape regular triangles $T$ with $h = \max_{T \in \mathcal T_h}
\operatorname{diam}(T)$, see,
e.g.,~\cite{Ciarlet:1978,BrennerScott:2008}. 

Now, we can define standard $H^1$-conforming finite element spaces,
where $\mathcal{P}_i$ denotes the space of polynomials of degree up to
$i$, 
\begin{align*}
\mathcal{Y}_h &:= \lbrace v \in H^1(D): v|_{T} \in \mathcal{P}_1(T) \text{
  for all } T\in \mathcal T_h \rbrace,\\
\mathcal{Y}_h^0 &:= \mathcal{Y}_h \cap \mathcal{Y}^0
\end{align*}
of piecewise linear finite elements.
For the controls, we choose a discretization of $\U$ by piecewise
constants, i.e., 
\begin{align*}
  \U_h &:= \lbrace u \in \U: v|_{T} \in \mathcal{P}_0(T) \text{
  for all } T\in \mathcal T_h \rbrace,\quad
  \Uad_h = \U_h \cap \Uad.
\end{align*}
Further, we define $P_h\colon \U \rightarrow \U_h$ as the
$L^2$-projection, i.e., for $v \in L^2(D)$ it is 
\[
  P_h(v)\bigl\lvert_T = \frac{1}{|T|}\int_T v\,\mathrm{d}x.
\]

Then the (spatially) discretized version of \eqref{eq:stationaryheatsourceproblemrandom} becomes
\begin{equation}\label{eq:stationaryheatsourceproblemrandom-discretized}
 \begin{aligned}
   \min_{u_h \in \Uad_h} \quad \Bigg\lbrace j_h(u_h)&:= \E[ J_h(u_h, \xi) ] = \E \left[\frac{1}{2} \lVert y_h  - y^D\rVert_{\U}^2 \right] + \frac{\lambda}{2} \lVert u_h \rVert_{\U}^2 \Bigg\rbrace
   \\
    \text{s.t. $\pP$-a.s.} \quad  b^\xi_h(y_h, v_h) &= (u_h,v_h)_{\U} \quad \forall v_h \in \mathcal{Y}_h^0.
 \end{aligned}
\end{equation}
Here $b^\xi_h$ is given by
\[
b^\xi_h(y,v) := \int_{D} I_ha(\xi) \nabla y\cdot \nabla
v \D x
\]
where $I_h$ is either the interpolation into
element wise constants or continuous linear finite elements. 
As it will be useful later, we state some well-known error estimates
for the interpolation.
As it will make calculations more easily accessible, we will use
so-called generic constants $c > 0$ which may have a different value
at each appearance but are independent of all relevant quantities.
\begin{lemma}\label{lem:interpolation-error-a}
  Given Assumption~\ref{assumption:regularity-random-field}, there
  exists a constant $C_r$ such that for almost every $\xi \in \Xi$, the expression
  \[
    \|a(\xi) - I_ha(\xi)\|_{L^\infty(D)} \le C_r h^t 
  \]
  is satisfied for both the interpolation by constants as well as the
  interpolation by piecewise linear functions.
\end{lemma}
\begin{proof}
We use a well-known interpolation estimate \cite[Theorem 4.4.20]{BrennerScott:2008} with $s =0$, $p = \infty$ and the cases $m=0$ and $m=1$ in combination with \cite[Section~4.5.2]{Triebel:1995} for the case of smooth domains and \cite[Example 1.9]{Lunardi:2018} for the case of convex polygons. Then, it follows that
\[
    \|a(\xi) - I_ha(\xi)\|_{L^\infty(D)} \le ch^t \|a(\xi)\|_{C^t(D)} 
  \]
  with the almost sure bound
  \[
    \|a(\xi)\|_{C^t(D)} \le \|a\|_{L^\infty(\Xi;C^t(D))}.
  \]
\end{proof}

It is then easy to see a representation of the gradient for the
reduced discretized functional $j_h \colon \U_h \rightarrow
\R$. Analogously to \eqref{eq:stochastic-gradient}, one obtains
\begin{lemma}\label{lemma:gradient-discretized}
  For $\xi \in \Xi$ and any $u_h \in \U_h$,
  the stochastic gradient $\nabla_u J_h(u_h,\xi) \in \U_h$ for problem \eqref{eq:stationaryheatsourceproblemrandom-discretized} is given by
$$\nabla_u J_h(u_h,\xi) = \lambda u_h - P_hp_h(\xi),$$
where $p_h(\xi) \in \mathcal{Y}^0_h$ solves the PDE
\begin{equation}\label{eq:Poissonrandom-adjoint-discretized}
b^\xi_h(v_h,p_h(\xi)) = (y^D - y_h(\xi),v_h)_{\U} \quad \forall v_h \in \mathcal{Y}^0_h
\end{equation}
and $P_h$ denotes the $L^2$-projection onto $\U_h$. 
\end{lemma}

We notice that $u_h \in \U_h \subset \U$ and thus one could simply
apply Algorithm~\ref{alg:PSG_Hilbert} to this discrete problem.
However, the gradient of $j$ at $u_h^n$ is 
\begin{align*}
  \nabla j(u_h^n) &=  \lambda u_h^n - \E[p^n(\xi)] \\
  &= \lambda u_h^n - \E[p^n(\xi)] \pm p^n(\xi^n) \pm P_h p_h^n(\xi^n)\\
  &=\underbrace{\lambda u_h^n - P_h p_h^n(\xi^n)}_{\nabla_u
    J_h(u_h^n,\xi^n)} + \underbrace{p^n(\xi^n) - \E[p^n(\xi)]}_{w^n} +
  \underbrace{P_h p_h^n(\xi^n) -  p^n(\xi^n)}_{r^n},
\end{align*}
highlighting that suitable mesh refinement needs to be added to assert that $r^n$ and thus $K_n = \esssup_{\omega \in \Omega} \lVert
r^n(\omega)\rVert_{\U}$ vanishes sufficiently fast in view
of the equations~\eqref{eq:Kstrongly-convex}, \eqref{eq:Kconvex-constant},
or~\eqref{eq:Kconvex-variable}.

To this end, we need to provide an estimate for
\[
  K_n =  \|P_h p_h^n(\xi) -
  p^n(\xi)\|_{L^\infty(\Xi,\U)}.
\]
In view of the $L^2(D) = \U$ stability of $P_h$ we have
\begin{equation}\label{eq:discretization-Kbound1}
\begin{aligned}
  \|P_h p_h^n (\xi) - p^n(\xi)\|_{\U} &\le \|P_h p_h^n(\xi)  -
  P_hp^n(\xi)\|_{\U}+\|P_h p^n(\xi) - p^n(\xi)\|_{\U}\\
  &\le \|p_h^n(\xi)-p^n(\xi)\|_{\U}+\|P_h p^n(\xi) -
  p^n(\xi)\|_{\U}\\
  & \le \|p_h^n(\xi)-p^n(\xi)\|_{\U}+ ch\|\nabla
  p^n(\xi)\|_{\U}
  \\
  & \le \|p_h^n(\xi)-p^n(\xi)\|_{\U}+ ch\Bigl(\|y^D\|_{\U} + \|u_h\|_{\U}\Bigr)
\end{aligned}
\end{equation}
using well-known error estimates for $P_h$ and the stability
estimate~\eqref{eq:Aprioribounds_Poisson_stochastic} for $p(\xi)$
and $y(\xi)$. To bound the first term on the right
of~\eqref{eq:discretization-Kbound1} we need a bit of preparation.

\begin{lemma}\label{lemma:feconvergence}
Under Assumption~\ref{assumption:regularity-random-field} there exists $s \in (0,1]$ and $c > 0$ such that 
  \begin{align*}
    \|y_h(\xi) - y(\xi)\|_{\U} &\le ch^{\min(2s,t)}\|u_h\|_{\U},\\
    \|p_h(\xi) - p(\xi)\|_{\U} &\le ch^{\min(2s,t)}\Bigl( \|y^D\|_{\U}+\|u_h\|_{\U}\Bigr)
  \end{align*}
  holds for almost every $\xi \in \Xi$.
\end{lemma}
\begin{proof}
  We split the error by introducing the intermediate function
  $y^h(\xi) \in \mathcal{Y}^0$ solving
  \[
    b^\xi_h(y^h(\xi),v) = (u,v)_{\U}\qquad \forall v \in \mathcal{Y}^0.
  \]
  Then to estimate $\| y_h(\xi) - y^h(\xi)\|_{\U}$, 
  we employ a standard duality argument (Aubin-Nitsche trick)
  using the uniform $H^{1+s}$-regularity of the problem, see
  Lemma~\ref{lemma:poissonrandom}, and obtain
  \[
    \|y_h(\xi) - y^h(\xi)\|_{\U} \le ch^{2s}\|u_h\|_{\U}.
  \]

  To estimate $\|y^h(\xi) - y(\xi)\|_{\U}$, we notice that
  $e = y^h(\xi) - y(\xi)$ solves the equation
  \[
    b^\xi(e,v) = ((a(\xi)-I_ha(\xi))\nabla y^h,\nabla v)_{\U}\quad
    \forall v \in \mathcal{Y}^0.
  \]
  In view of Lemma~\ref{lemma:poissonrandom}, it is sufficient to
  estimate the $H^{-1}$-norm of the right-hand side $f =-\nabla \cdot
  ((a(\xi)-I_ha(\xi))\nabla y^h(\xi))$.
  It is immediately clear by definition, and
  Lemma~\ref{lemma:poissonrandom}, that 
  \begin{align*}
    \|f\|_{H^{-1}(D)} &\le \|\nabla
    y^h(\xi)\|_{\U}\|a(\xi)-I_ha(\xi)\|_{L^\infty(D)}\\
    &\le C_0C_r\|u_h\|_{\U}h^t,
  \end{align*}
  showing
  \[
    \|y^h(\xi) - y(\xi)\|_{\U} \le c h^t \|u_h\|_{\U}.
  \]
The triangle inequality yields the estimate for $y_h(\xi)-y(\xi)$.
  
  Analogous calculations give the estimate for $p_h(\xi)-p(\xi)$.
\end{proof}

Combining Lemma~\ref{lemma:feconvergence}
with~\eqref{eq:discretization-Kbound1}, we obtain the bound
\begin{equation}\label{eq:Kn-bound}
  K_n \le c h^{\min(2s,t,1)}\Bigl(\|y^D\|_{\U}+\|u_h\|_{\U}\Bigr).
\end{equation}
From this it is easy to derive relations for the selection of the mesh size $h_n$ in the $n^{\text{th}}$ iteration based on the estimates obtained in section~\ref{sec:PSG-Efficiency} and the bound \eqref{eq:Kn-bound}. 

For the strongly convex case, \eqref{eq:Kstrongly-convex} implies that we need for a fixed $K>0$
$$c h^{\min(2s,t,1)}\Bigl(\|y^D\|_{\U}+\|u_h\|_{\U}\Bigr) \leq \frac{K}{n+\nu}.$$
We note that the strongly convex parameter for
\eqref{eq:stationaryheatsourceproblemrandom} is $\mu=\lambda$. From
Theorem~\ref{thm:strongly-convex-step-size-bias-rule}  we get with
$\theta > 1/(2\lambda)$ and $\nu \geq 2\theta K/(2\lambda \theta -1)
-1$ the rule
\begin{equation}
 \label{eq:strongly-convex-bias-stepsize-rule}
 t_n = \frac{\theta}{n+\nu}, \quad  h_n \le \Bigl(\frac{c}{n+\nu}\Bigr)^{1/\min(2s,t,1)}.
\end{equation}

For the convex case with constant step sizes, from \eqref{eq:Kconvex-constant} we have the requirement that 
\begin{equation}
\label{eq:bound-convex-case-constant-step}
 \sum_{n=1}^N c h_n^{\min(2s,t,1)} \propto \sqrt{N}.
\end{equation}
Thus we get from  \eqref{eq:stepsizerule-convex-constant} and \eqref{eq:bound-convex-case-constant-step} the rule
\begin{equation}
\label{eq:convex-bias-constant-stepsize-rule}
t = \frac{\Dad}{\sqrt{M N}}, \quad h_n \leq \left(c (\sqrt{n} - \sqrt{n-1})\right)^{1/\min(2s,t,1)}.
\end{equation}
For the convex case with variable step sizes, choosing $i = \lceil \alpha N \rceil$ for a fixed $\alpha \in (0,1)$, \eqref{eq:Kconvex-variable} requires 
\begin{equation}\label{eq:Kconvex-variable-new}
  \sum_{n=i}^N \frac{1}{\sqrt{n}}c h_n^{\min(2s,t,1)}\propto 1.
\end{equation}
Therefore with a similar argument, we get for a constant $\theta > 0$
\begin{equation}
  \label{eq:convex-bias-variable-stepsize-rule}
  \begin{aligned}
    t_n &= \frac{\theta \Dad}{\sqrt{M n}}, \\
    h_n &\leq \left(c (\sqrt{n}
      - \sqrt{n-1})\right)^{1/\min(2s,t,1)} \\
    &= \left(\frac{c}{\sqrt{n}+\sqrt{n-1}}\right)^{1/\min(2s,t,1)}.
  \end{aligned}
\end{equation}

Summarizing, by suitable control of the mesh size, and thus the
discretization bias, we can recover the convergence rates proven
in~Theorem~\ref{thm:strongly-convex-step-size-bias-rule} and
~Theorem~\ref{thm:convex-averaging} as follows:
\begin{thm}
\label{thm:main-result}
If $j$ is $\mu$-strongly convex, $\theta$ and $\nu$ are chosen such that $\theta > 1/(2\mu)$ and $\nu \geq 2\theta K/(2\mu \theta -1) -1$, and step sizes and mesh fineness is chosen to satisfy \eqref{eq:strongly-convex-bias-stepsize-rule}, then 
\begin{equation*}
\label{eq:efficiency-iterates-strongly-convex-mesh}
 \E[\lVert u^n - u \rVert_{\U}] \leq \mathcal{O}\left(\frac{1}{\sqrt{n+\nu}}\right).
\end{equation*}
If $j$ is $\mu$-strongly convex, $\nabla j$ is Lipschitz continuous, and $\nabla j(u)=0$, then
$$\E[j(u^n)-j(u)] \leq  \mathcal{O}\left(\frac{1}{n+\nu}\right).$$

If $j$ is generally convex, and step sizes and mesh fineness are chosen to satisfy \eqref{eq:convex-bias-constant-stepsize-rule}, then
$$\E[j(\tu_1^N)-j(u)]  \leq  \mathcal{O}\left(\frac{1}{\sqrt{N}}\right).$$

Finally, if $j$ is generally convex, and step sizes and mesh fineness are chosen to satisfy \eqref{eq:convex-bias-variable-stepsize-rule}, then 
$$\E[j(\tilde{u}_i^N) - j(u)] \leq \mathcal{O}\left(\frac{1}{\sqrt{N}}\right)$$
as long as $i = \lceil\alpha N\rceil$ for some $\alpha \in (0,1)$.
\end{thm}

\begin{proof}
This result immediately follows from Theorem~\ref{thm:strongly-convex-step-size-bias-rule} and Theorem~\ref{thm:convex-averaging}.
\end{proof}

Theorem~\ref{thm:main-result} allows for an a priori coupling of the mesh refinement with the progress of the projected stochastic gradient method, and we obtain the discretized version of Algorithm~\ref{alg:PSG_Hilbert}. The
resulting algorithm is given in
Algorithm~\ref{alg:PSG_Hilbert_discretized}.
\begin{algorithm}
\begin{algorithmic}[1] 
\STATE \textbf{Initialization:} Select $h_1 > 0$, $u_h^1 \in \Uad_{h}$
\FOR{$n=1,2,\dots$}
\STATE Generate $\xi^n$, independent from $\xi^1, \dots, \xi^{n-1}$,
and $t_n >0$, $K_n > 0$
\IF{$h=h_n$ is too large per \eqref{eq:strongly-convex-bias-stepsize-rule}, \eqref{eq:bound-convex-case-constant-step}, or \eqref{eq:convex-bias-variable-stepsize-rule}}
\STATE Refine mesh $\mathcal T_{h_n}$ until $h=h_n$ is small enough.
\ENDIF
\STATE Calculate $(y_h^n,p_h^n)$ solving
\[
  b^{\xi^n}(y_h^n,v_h) = (u_h^n,v_h)_{\U}, \quad b^{\xi^n}(v_h,p_h^n) = (y^D-y_h^n,v_h)_{\U}
\]
for all $v_h \in \mathcal{Y}_h^0$.
\STATE $u_h^{n+1} := \pi_{\Uad_h}(u_h^n - {t_n}(\lambda u_h^n - P_hp_h^n))$
\ENDFOR
\end{algorithmic}
\caption{Projected Stochastic Gradient (PSG) - Discretized Version}
\label{alg:PSG_Hilbert_discretized}
\end{algorithm}
Let us note that in both cases the scaling of the mesh size
parameters $h_n$ is identical, and boundedness
of~\eqref{eq:Kconvex-variable} follows by the particular choice $i =
\lceil\alpha N\rceil$ since then
\[
  h_n^{\min(2s,t,1)} \le \frac{c}{\sqrt{n}}
\]
and consequently
\[
  \sum_{n=i}^N \frac{h_n^{\min(2s,t,1)}}{\sqrt{n}} \le c\sum_{n=i}^N
  \frac{1}{n}
  \le c \frac{(N-i+1)}{i} \le c \frac{(1-\alpha) N + 1}{\alpha N}
  \rightarrow c 
\]
as $(N \rightarrow \infty)$.

\begin{remark}\label{rem:meshrefinement-lower-regularity}
  While in some situations, the constant $s$ can be calculated, in general it
  is unknown. Hence it appears to be natural to guess, probably
  mistakenly, that $\min(2s,t,1) = 1$. Now, for large values of $n$
  \[
    \frac{c}{\sqrt{n}+\sqrt{n-1}} < 1
  \]
  and thus
  \begin{align*}
    \frac{c}{\sqrt{n}+\sqrt{n-1}} &\ge
    \frac{c}{\sqrt{n}+\sqrt{n-1}}
    \left(\frac{c}{\sqrt{n}+\sqrt{n-1}}\right)^{1/\min(2s,t,1)-1}\\
    &=\left(\frac{c}{\sqrt{n}+\sqrt{n-1}}\right)^{1/\min(2s,t,1)}.
  \end{align*}
  Consequently, having $h_n \simeq \frac{c}{\sqrt{n}+\sqrt{n-1}}$
  while $\min(1,2s,t) = p < 1$
  will give
  \[
    h_n^{\min(2s,t,1)} \simeq \frac{1}{(n+\nu)^{p}} \gg \frac{1}{n+\nu},
  \]
  slowing the convergence of the algorithm.
  An analogous argument can be made for the rule \eqref{eq:strongly-convex-bias-stepsize-rule}. 
\end{remark}

\begin{remark}
  Note that our above coupling does not require the mesh to be
  uniform, i.e., it is possible that $\min_{T \in \mathcal T_h} h_T
  \ll \min_{T \in \mathcal T_h} h_T $. This allows to handle
  singularities in the problem, e.g., boundary values or jumping
  coefficients by suitably graded meshes.

  Further, for a reliable a posteriori error estimator $\eta(\xi)$,
  i.e., for some $c$ independent of $h$ and $\xi$ it holds
  \[
    \|a(\xi)-I_ha(\xi)\|_{L^\infty(D)}+\|y_h(\xi)-y(\xi)\|_{\mathcal U}+\|p_h(\xi) - p(\xi)\|_{\U} \le c \eta(\xi),
  \]
  one can easily obtain an analogous coupling between $n$ and a
  tolerance for $\eta(\xi)$  by replacing $h^{\min(2s,t,1)}$ with
  $\eta(\xi_n)$ in~\eqref{eq:bound-convex-case-constant-step}
  and~\eqref{eq:Kconvex-variable-new}.
  Then the a posteriori controlled version of the algorithm is
  immediately obtained replacing line~5 in Algorithm~\ref{alg:PSG_Hilbert_discretized}
  by \emph{Refine mesh $\mathcal T_{h_n}$ until $\eta(\xi^n)$ is small
  enough}.
\end{remark}
\section{Numerical Experiments}\label{sec:Experiments}
Let the domain be given by $D=(0,1)\times(0,1)$ and $\Uad = \{ u \in \U \,|\, -1 \leq u(x) \leq 1  \quad \forall x \in D\}.$ For all simulations, we choose $u^1 \equiv 0$. For the strongly convex case, we define
$y^D(x) = -\left(8 \pi^2 + \frac{1}{8 \pi^2 \lambda} \right)\sin(2 \pi x_1) \sin(2 \pi x_2)$. For the convex case, we use $\lambda =0$ and the following modified PDE constraint
\begin{align}
\label{eq:modified-PDE-constraint}
 \quad -\nabla \cdot (a(x,\xi) \nabla y(x,\xi)) &= u(x)+e^D(x), \quad (x,\xi) \in D \times \Xi\\
 y(x,\xi) &=0,\phantom{(x)+e^D(x),}  \quad(x,\xi) \in \partial D \times \Xi.
\end{align}
with $y^D(x) = \sin(\pi x_1) \sin(\pi x_2)+3\sin(2 \pi x_1) \sin(2\pi x_2)$ and the function $e^D(x) =6 \pi^2 \sin(\pi x_1) \sin(\pi x_2) - \text{sign}(\sin(2\pi x_1) \sin(2\pi x_2))$.

\subsection{Random Field Choices}
To demonstrate the effect of the random field choice on the convergence, we observe three different random fields. Example realizations of the fields are shown in Figure~\ref{fig:random-field-realizations}. We recall that for a random field $a$, the Karhunen--Lo\`eve expansion takes the form 
\begin{equation}
\label{eq:KL-expansion}
 a(x,\omega) = a_0 + \sum_{i=1}^\infty \sqrt{\lambda_i} \phi_i(x) \xi_i(\omega),
\end{equation}
where $\xi_i$ is a random variable with given probability distribution, and $\lambda_i$ and $\phi_i$ denote the eigenvalues and eigenfunctions associated with the compact self-adjoint operator defined via the covariance function $C \in L^2(D\times D)$ by
$$\mathcal{C}(\phi)(x) = \int_D C(x,y) \phi(y) \D y, \quad x\in D.$$ For simulations, we use a finite dimensional noise assumption to replace \eqref{eq:KL-expansion} with
\begin{equation}
\label{eq:truncated-KL-expansion}
 a(x,\xi) = a_0 + \sum_{i=1}^m \sqrt{\lambda_i} \phi_i(x) \xi_i(\omega).
\end{equation}
For an interval $[a,b]$ where $a<b$, we denote the uniform distribution by $U(a,b)$ and the truncated normal distribution with parameters $\mu$ and $\sigma$ by $\mathcal{N}(\mu, \sigma, a, b)$ \footnote{The parameters $\mu$ and $\sigma$ correspond to the mean and standard deviation of the standard normal distribution $N(\mu, \sigma)$; see  \cite{Kroese2011} for a definition of the truncated distribution.}
\begin{remark}
Of course, truncating the Karhunen--Lo\`eve expansion after
$m$ summands will introduce an additional error, in general.
This can be included in the error estimates in
Lemma~\ref{lemma:feconvergence} analogous to the error
in the uncertain coefficient due to interpolation.
\end{remark}
\paragraph{Example 1}
For the first example (cf.,~\cite[Example 9.37]{Lord2014}), we choose $a_0 = 5,  m = 20$, and $\xi_i \sim U(-\sqrt{3},\sqrt{3})$ for $i=1, \dots, m$. The eigenfunctions and eigenvalues are given by 
$$\tilde{\phi}_{j,k}(x):= 2\cos(j \pi x_2)\cos(k \pi x_1), \quad \tilde{\lambda}_{k,j}:=\frac{1}{4} \exp(-\pi(j^2+k^2)l^2), \quad j,k \geq 1,$$
where we reorder terms so that the eigenvalues appear in descending order (i.e., $\phi_1 = \tilde{\phi}_{1,1}$ and $\lambda_1 = \tilde{\lambda}_{1,1}$) and we choose the correlation length $l=0.5$.
\paragraph{Example 2}
For the second example, we generate a log-normal random field with truncated Gaussian noise by first generating a truncated expansion for a Gaussian field with a separable exponential, i.e.,~the covariance function has the form
$$C(x,y) = e^{-|x_1-y_1|/l_1 -|x_2-y_2|/l_2}$$
on $D=[-\tfrac{1}{2},\tfrac{1}{2}]^2.$ The eigenfunctions are given by $\phi_j(x) = \phi_{i,1}(x_1)\phi_{k,2}(x_2)$ and the eigenvalues are $\lambda_j = \lambda_{i,1} \lambda_{k,2}$, where $\phi_{i,m}, \lambda_{i,m}$ are for $m=1,2$ solutions to 
\begin{equation}
\label{eq:eigenvalue-problem-gauss-RF}
\int_{-1/2}^{1/2} e^{-|x_m-y_m|/l_m} \phi^m(y_m) \D y_m = \lambda^m \phi^m(x_m), \quad x_m \in [-\tfrac{1}{2}, \tfrac{1}{2}].
\end{equation}
Solutions to \eqref{eq:eigenvalue-problem-gauss-RF} have the analytic expression (cf.,~\cite[Example 7.55]{Lord2014})
\begin{equation}
 \label{eq:eigenfunctions-eigenvalues-Gauss}
\begin{aligned}
 \phi^{i,m} &= \begin{cases}
               \sqrt{1/2+\sin(\omega_i)/(2 \omega_i)}^{-1} \cos(\omega_i x_m), & i \text{ odd}\\ 
                \sqrt{1/2-\sin(\omega_i)/(2 \omega_i)}^{-1} \sin(\omega_i x_m), & i \text{ even} 
              \end{cases}\\
\lambda_{i,m} &= \frac{2 l_m^{-1}}{\omega_i^2 + l_m^{-2}}, \quad \quad \quad  \omega_i = \begin{cases}
             \hat{\omega}_{\lceil{i/2}\rceil}, & i \text{ odd}\\
             \tilde{\omega}_{i/2}, & i \text{ even}
            \end{cases},
\end{aligned}
\end{equation}
where $\hat{\omega}_j$ is the $j^{\text{th}}$ positive root of $l^{-1}-\omega \tan(\omega/2)$ and $\tilde{\omega}_j$ is the $j^{\text{th}}$ positive root of $l^{-1} \tan(\omega/2) + \omega.$ Sorting terms in \eqref{eq:eigenfunctions-eigenvalues-Gauss} by decreasing eigenvalues and reindexing, we define the log-normal field with truncated Gaussian noise by
\begin{equation}
\label{eq:lognormal}
a(x,\xi) = e^{a_0+\sum_{i=1}^m \sqrt{\lambda_i} \phi_i(x) \xi_i(\omega)} 
\end{equation}
with $a_0=1$, $l_1=l_2=1$, $m=100$, and $\xi_i \sim \mathcal{N}(0,0.1,-100,100).$ In simulations, the random fields are additionally transformed to $(0,1)\times(0,1)$. For this choice, the trajectories of $a$ belong to $C^t(\bar{D})$ for all $t<1/2$; see \cite[Lemma 2.3]{Charrier2013}.

\paragraph{Example 3}
We observe an example that does not satisfy Assumption~\ref{assumption:regularity-random-field}. We partition $D$ into two non-overlapping subdomains $D_1$, $D_2$ and define a piecewise constant field by
\begin{equation}
\label{eq:RF3}
a(x,\omega) = \xi_1(\omega) \mathds{1}_{D_1}(x) + \xi_2(\omega) \mathds{1}_{D_2}(x)
\end{equation}
where $\mathds{1}_{D_i}$ is the indicator function of the set $D_i \subset D$ and $\xi_i$ are bounded, positive and independent random variables. In simulations, we let $D_1 = (0,1) \times (1/2,1)$ and $D_2 = (0,1) \times (0,1/2)$; we let $\xi_1 \sim U(3, 4)$ and $\xi_2 \sim U(1, 2)$.

\begin{figure}
\centering
  \begin{subfigure}[b]{0.32\linewidth}
  \centering
    \includegraphics[height = 3.7cm]{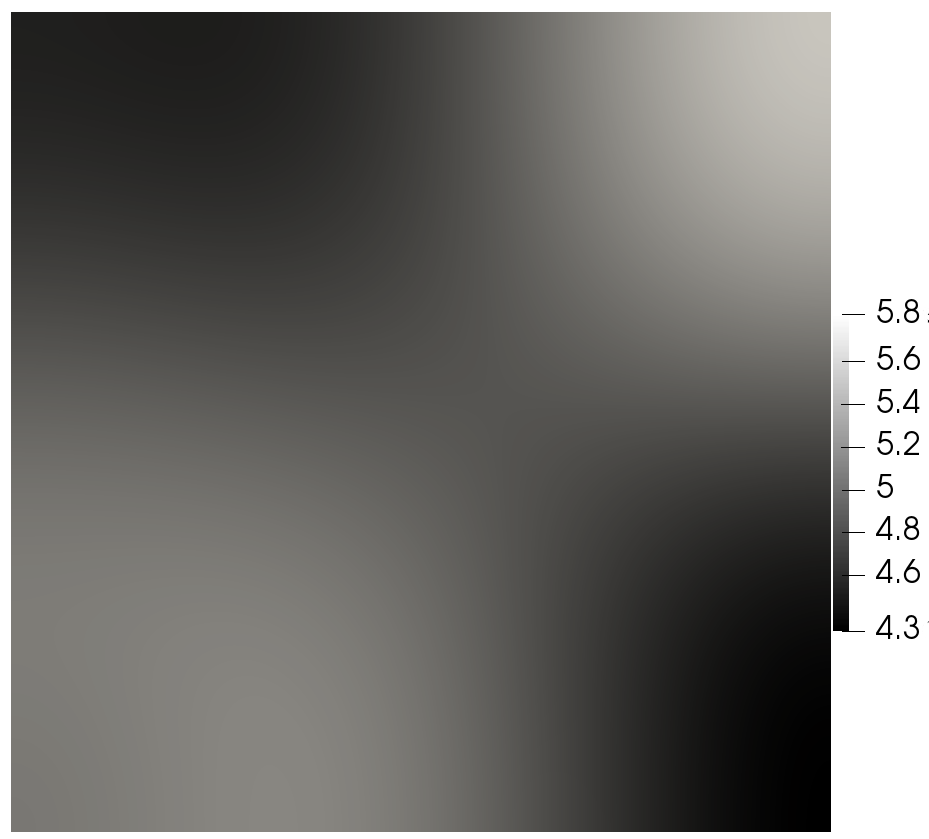}
    \caption{Example 1}  
  \end{subfigure}
  \begin{subfigure}[b]{0.32\linewidth}
  \centering
    \includegraphics[height = 3.7cm]{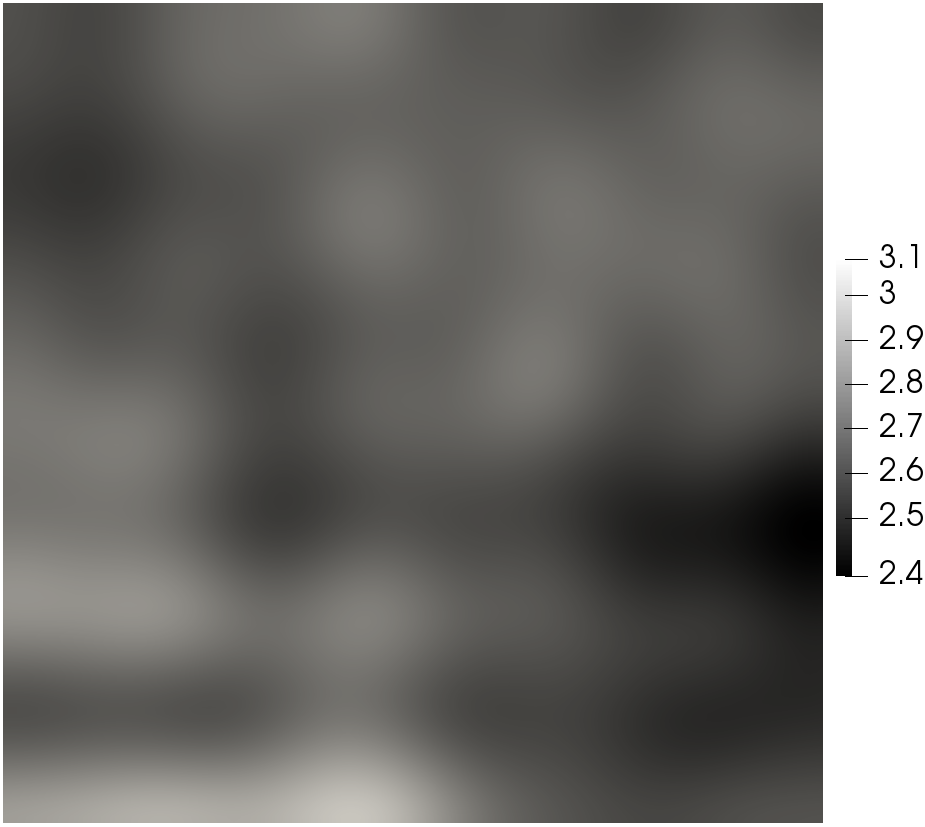}
    \caption{Example 2}  
  \end{subfigure}
   \begin{subfigure}[b]{0.32\linewidth}
  \centering
    \includegraphics[height = 3.7cm]{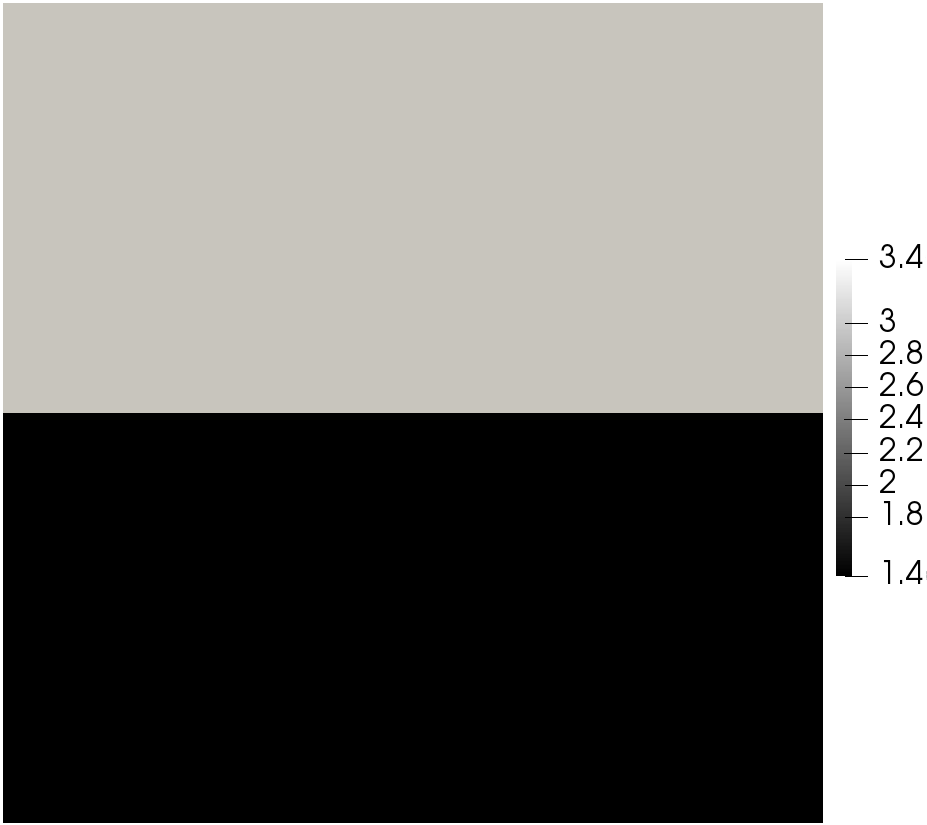}
    \caption{Example 3}  
  \end{subfigure}
  \caption{Single realizations of each random field.}
  \label{fig:random-field-realizations}
\end{figure}

\subsection{Experiments}
Simulations were run on FEniCS \cite{Alnes2015} on a laptop with Intel
Core i7 Processor (8 x 2.6 GHz) with 16 GB RAM. 
In all experiments, the initial mesh contained eight triangles and was uniformly refined using newest vertex bisection.

\paragraph{Effect of mesh refinement on objective function value}
In the first experiment, we observe objective function values with and without mesh refinement for the random field in example 1. The strongly convex case is observed with $\lambda = 0.1.$  A total of 1000 samples is taken at iteration $n=100$ and objective function values are compared. We use step sizes \eqref{eq:strongly-convex-bias-stepsize-rule} where $\theta = 1/(2\lambda) +1$, $\nu = 2\theta K/(2\lambda \theta -1) -1$ and $K=5$. Without refinement, where the mesh is constant $h \approx 0.18$, $\hat{j}_{100} \approx 779.503$. With refinement, where the mesh is refined according to \eqref{eq:strongly-convex-bias-stepsize-rule}, we get  $h_{100} \approx 0.04$ and $\hat{j}_{100} \approx 779.479$. Figure \ref{fig:experiment1} shows clear jumps where the mesh is refined. 

\begin{figure}
 \centering
  \begin{subfigure}[b]{0.45\linewidth}
   \centering
  \includegraphics[height = 4.0cm]{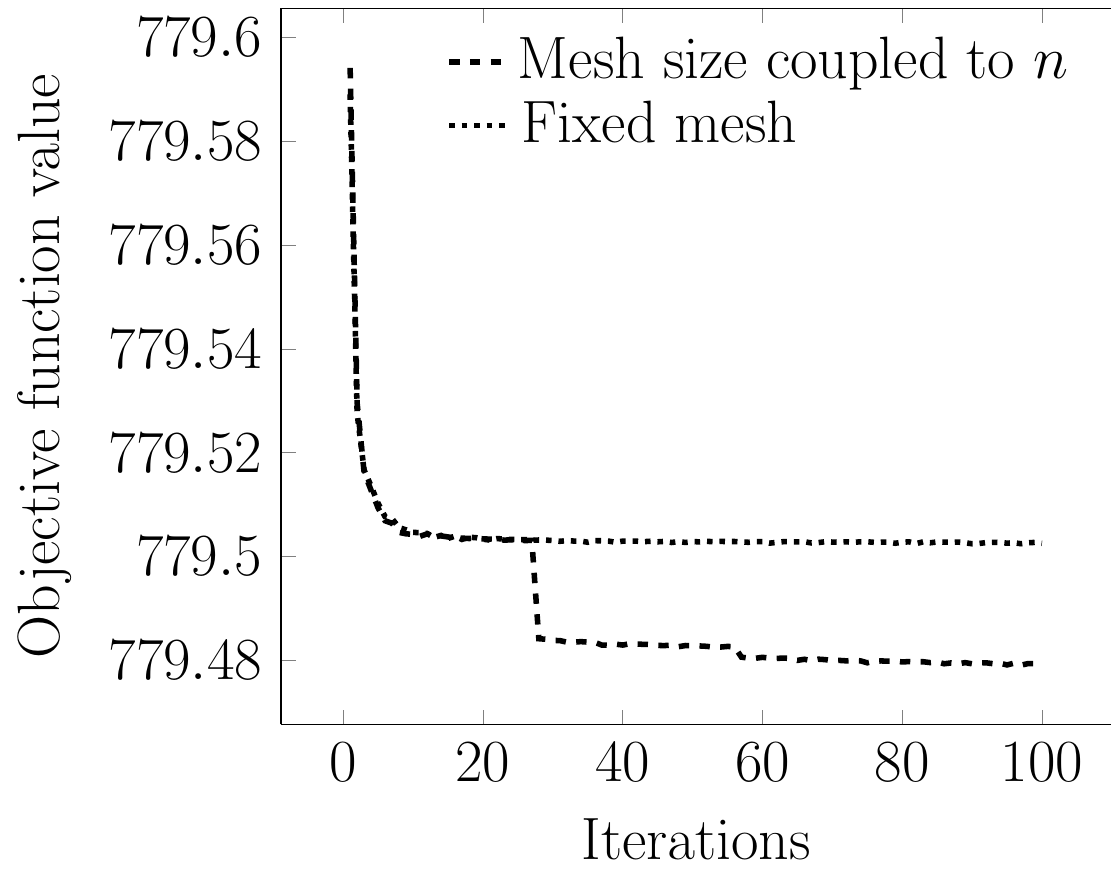}
  \caption{Objective function}
    \end{subfigure}
  \begin{subfigure}[b]{0.45\linewidth}
  \centering
    \includegraphics[height = 4.0cm]{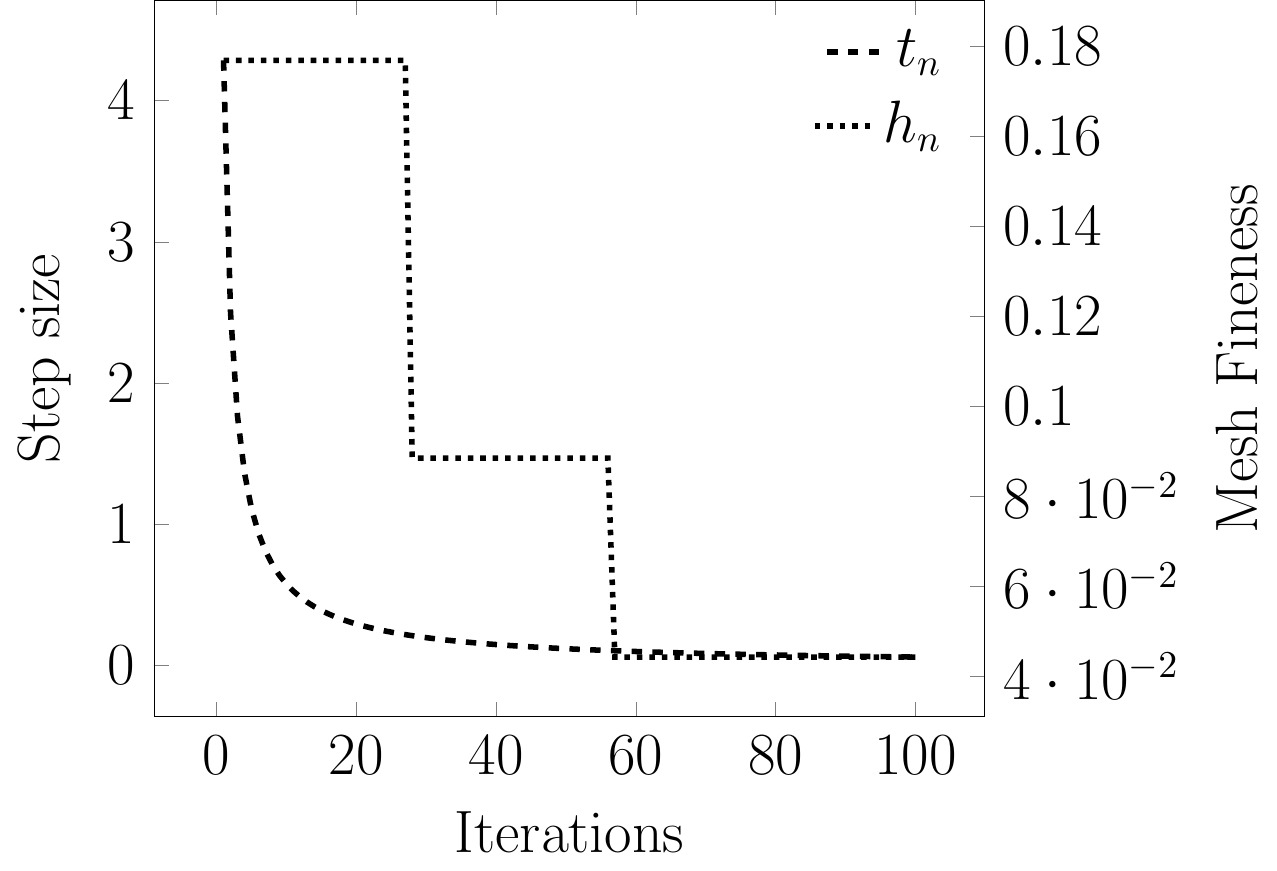}
    \caption{Step size and mesh fineness}  
  \end{subfigure}
  \caption{Behavior of objective function with and without mesh refinement.}
      \label{fig:experiment1}
\end{figure}

\paragraph{Convergence plots - Strongly Convex Case}
To demonstrate Algorithm~\ref{alg:PSG_Hilbert_discretized} using \eqref{eq:strongly-convex-bias-stepsize-rule}, we choose the example for the strongly convex case with $\lambda = 0.2$, $\theta > 1/(2\lambda) +1$, $K=1$, and $\nu = 2 \theta K/(2\lambda \theta -1) -1$, and finally, $c = 17.5$, which was chosen to prevent the mesh from refining too aggressively. To generate reference solutions, the algorithm was run for $n=3000$ iterations with $h_{1000} \approx 0.0044$ to get $\bar{u}:=u_{3000}$; these solutions are shown for each of the random fields in Figure~\ref{fig:reference solutions-SC}.

\begin{figure}
\centering
  \begin{subfigure}[b]{0.32\linewidth}
  \centering
    \includegraphics[height = 3.4cm]{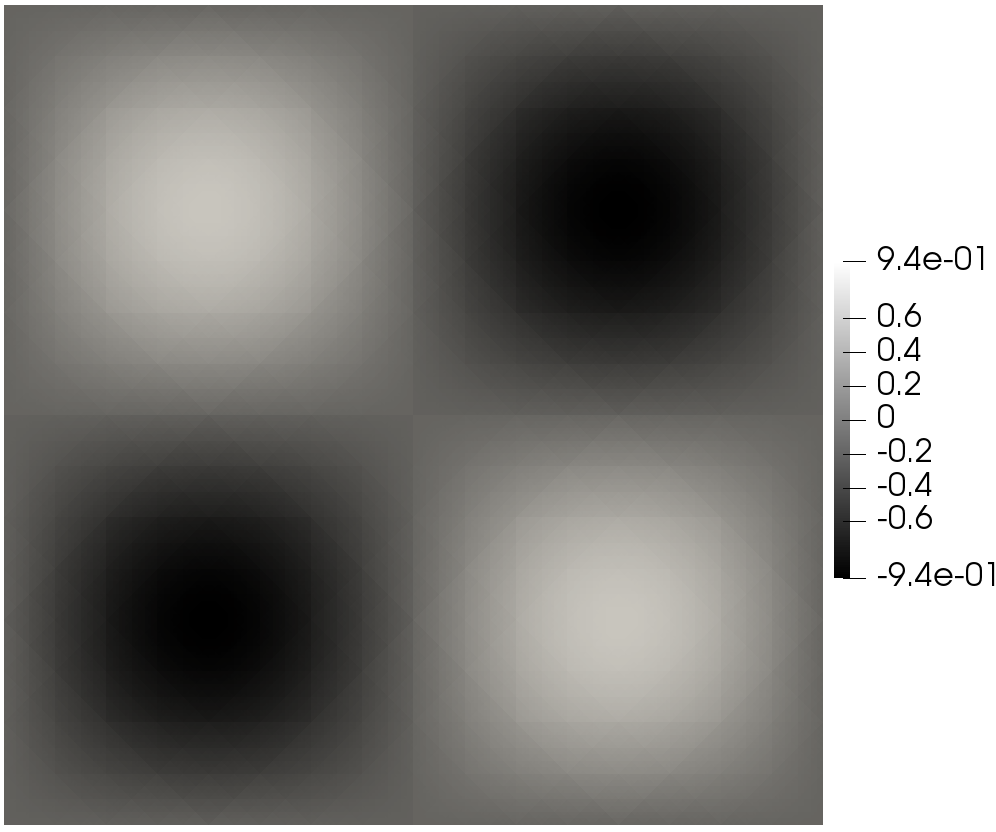}
    \caption{Example 1}  
  \end{subfigure}
  \begin{subfigure}[b]{0.32\linewidth}
  \centering
    \includegraphics[height = 3.4cm]{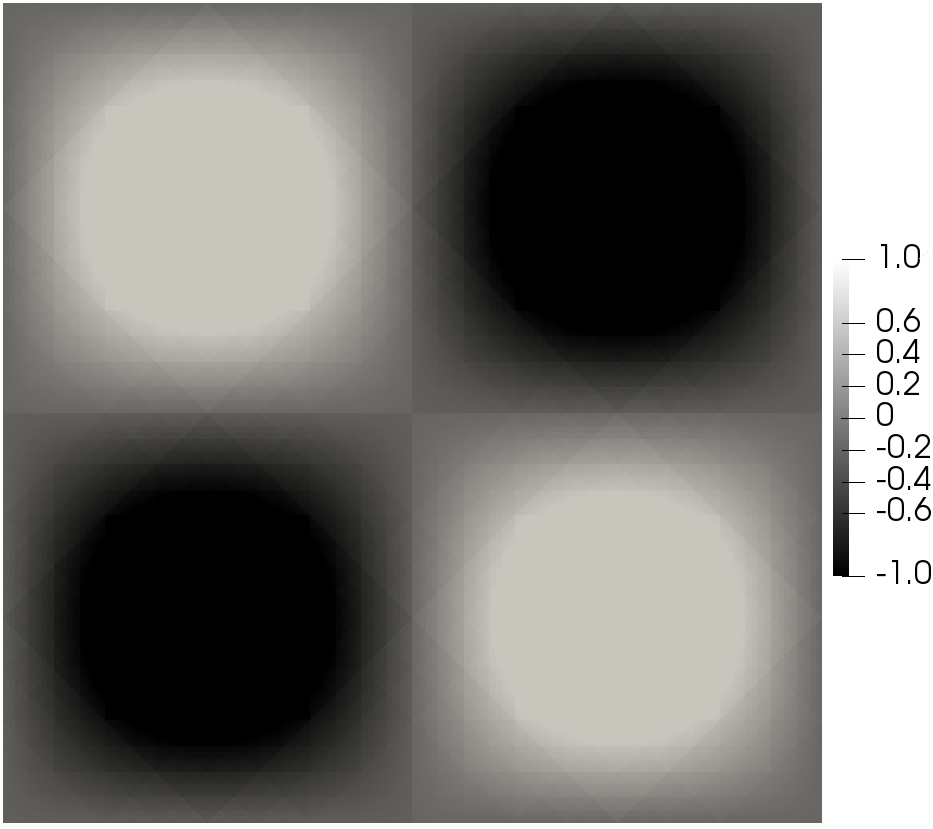}
    \caption{Example 2}  
  \end{subfigure}
   \begin{subfigure}[b]{0.32\linewidth}
  \centering
    \includegraphics[height = 3.4cm]{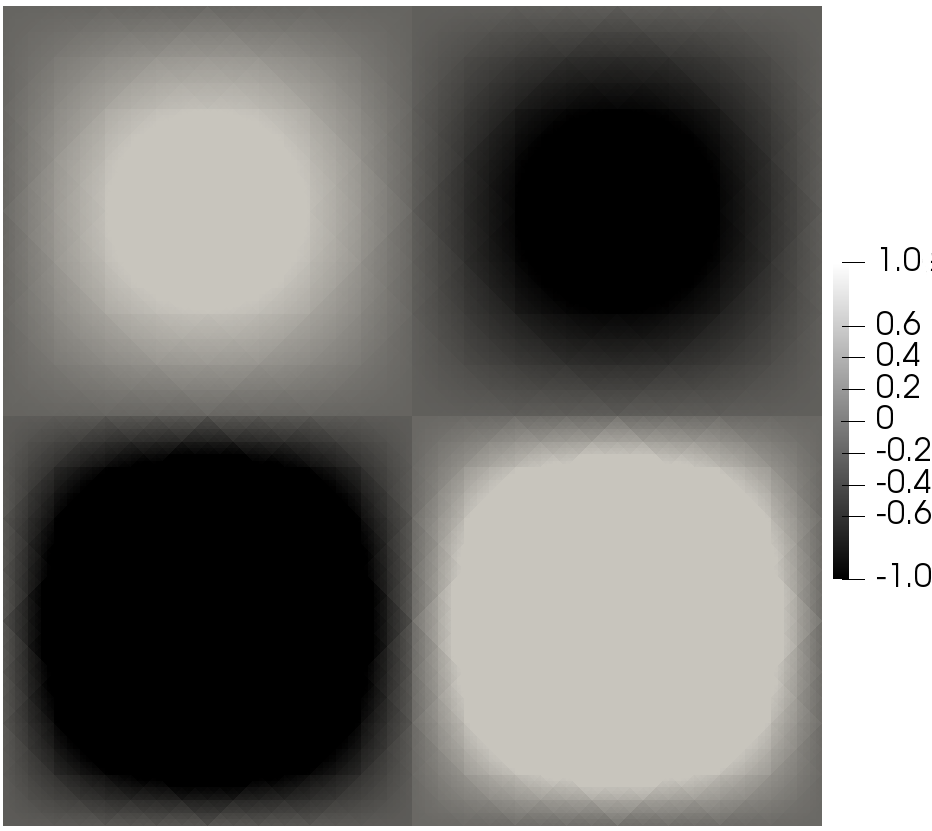}
    \caption{Example 3}  
  \end{subfigure}
  \caption{Reference solutions for strongly convex case.}
  \label{fig:reference solutions-SC}
\end{figure}

We observe behavior of the algorithm for a single run with $300$ iterations. To approximate objective function values, $m=1000$ samples are generated to get
$\hat{j}(u_h^n) = \tfrac{1}{m} \sum_{i=1}^m J(u_h^n,\xi^{n,i})$, where
$\xi^{n,i}$ denotes a newly generated $i^{\text{th}}$ sample at
iteration $n$. We set $\hat{\bar{j}}:=\hat{j}(u_h^{3000}).$ We observe
objective function decay and convergence rates $\lVert u_h^n - \bar{u}
\rVert_{\U}$  and $|\hat{j}(u_h^n) - \hat{\bar{j}}|$ for a single run
of the algorithm for each of the random fields; see
Figure~\ref{fig:SC-RF-1-convergence}--\ref{fig:SC-RF-3-convergence}. 
To approximate $\lVert u_h^n - \bar{u} \rVert_{\U}$, we project $u_h^n$ onto the fine mesh used for $\bar{u}$ and compute the error on the fine mesh. In each example, we see clear jumps in the objective function value when the mesh is refined, followed by decay at or better than the expected rate. A single run of 1000 iterations with mesh refinement took 36\% of the CPU time when compared to computations on a fixed mesh (corresponding to $h_{1000} \approx 0.011$).

\begin{figure}
\centering
  \begin{subfigure}[b]{0.32\linewidth}
  \centering
    \includegraphics[height = 2.8cm]{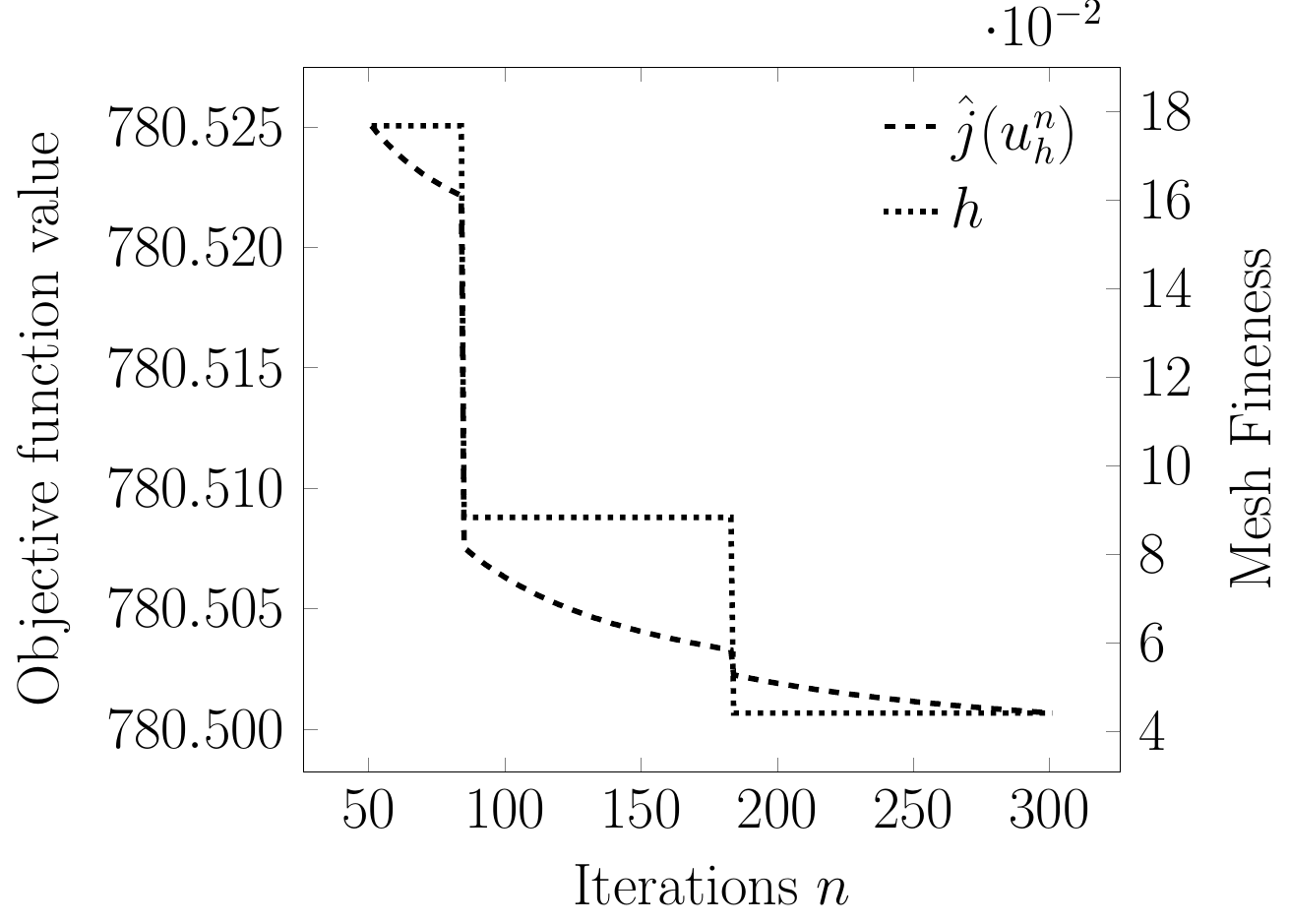}
    \caption{Objective function}  
  \end{subfigure}
   \begin{subfigure}[b]{0.32\linewidth}
  \centering
    \includegraphics[height = 2.8cm]{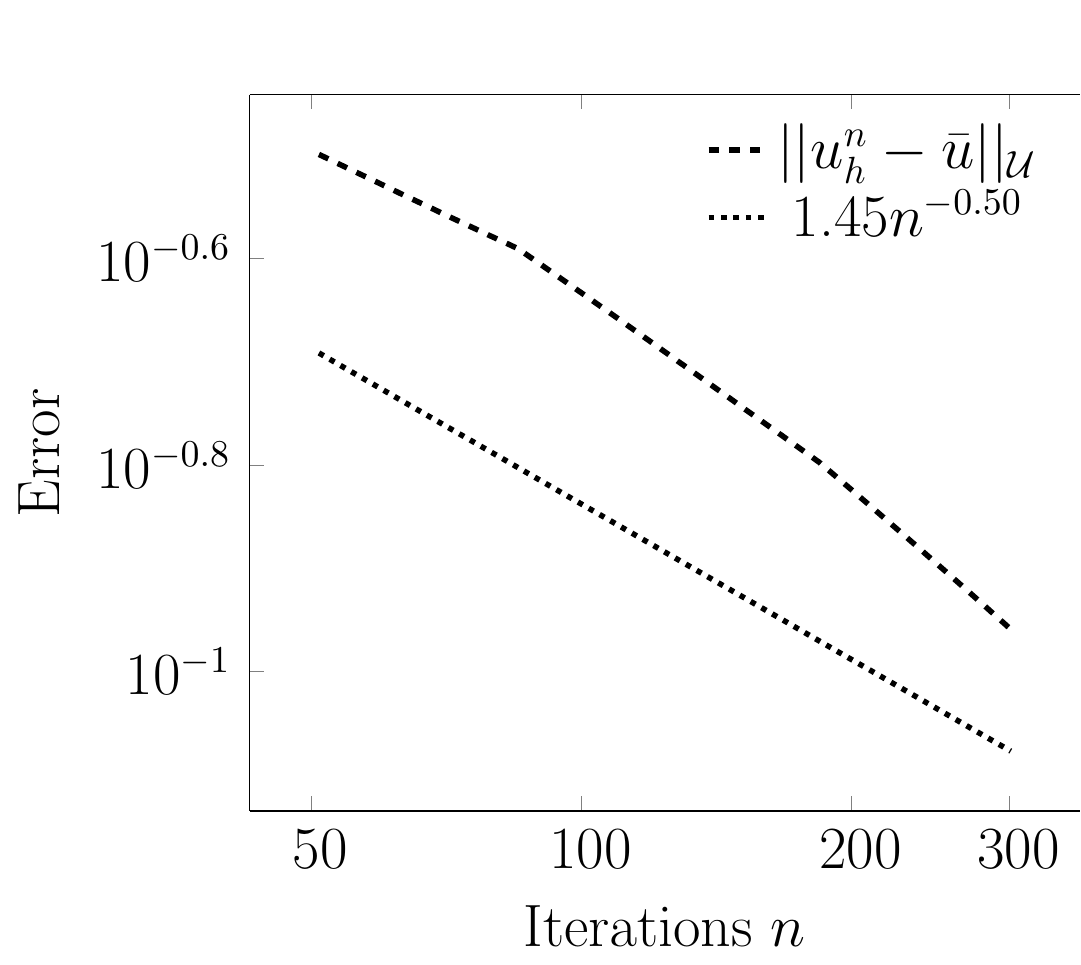}
    \caption{Error in iterates}  
  \end{subfigure}
 \begin{subfigure}[b]{0.32\linewidth}
  \centering
    \includegraphics[height = 2.8cm]{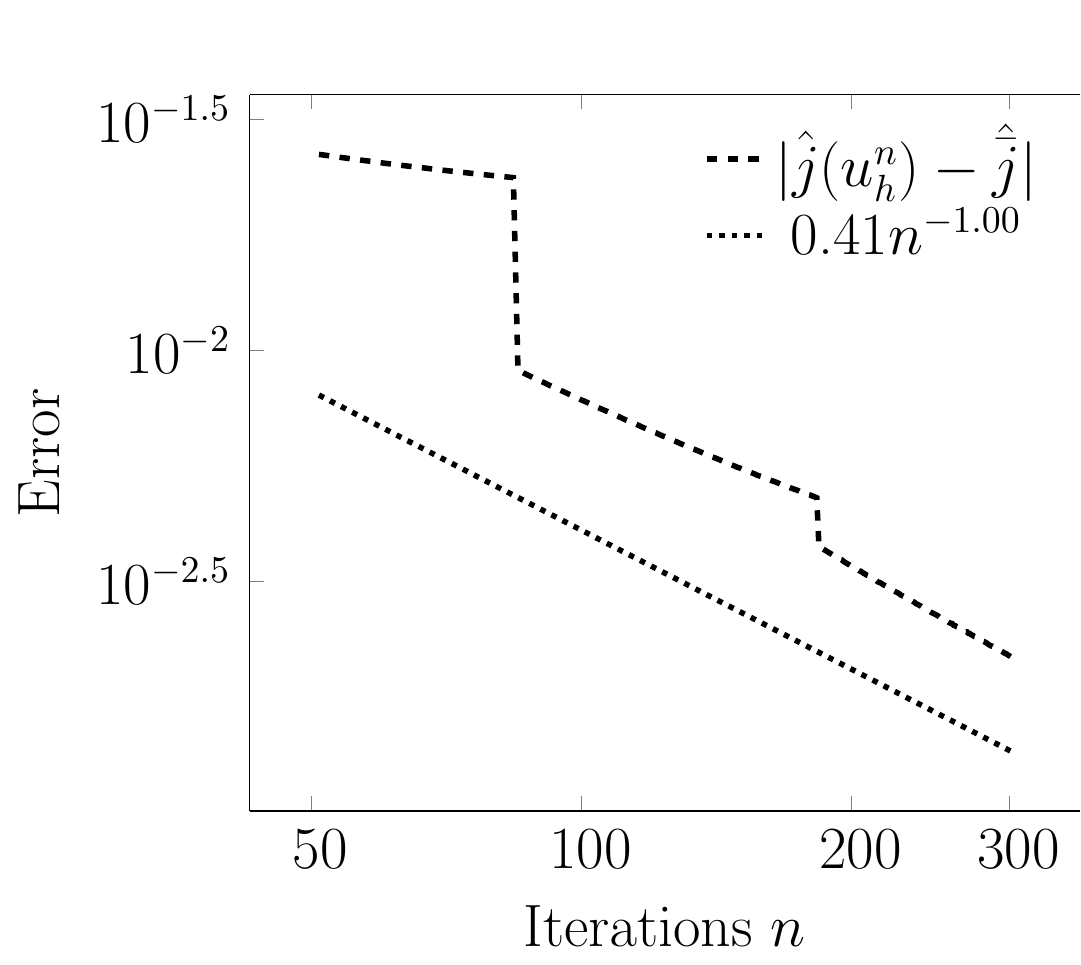}
    \caption{Error in objective function}  
  \end{subfigure}
  \caption{Strongly convex functional with smooth random field
    (example 1).}
  \label{fig:SC-RF-1-convergence}
\end{figure}

\begin{figure}
\centering
  \begin{subfigure}[b]{0.32\linewidth}
  \centering
    \includegraphics[height = 2.8cm]{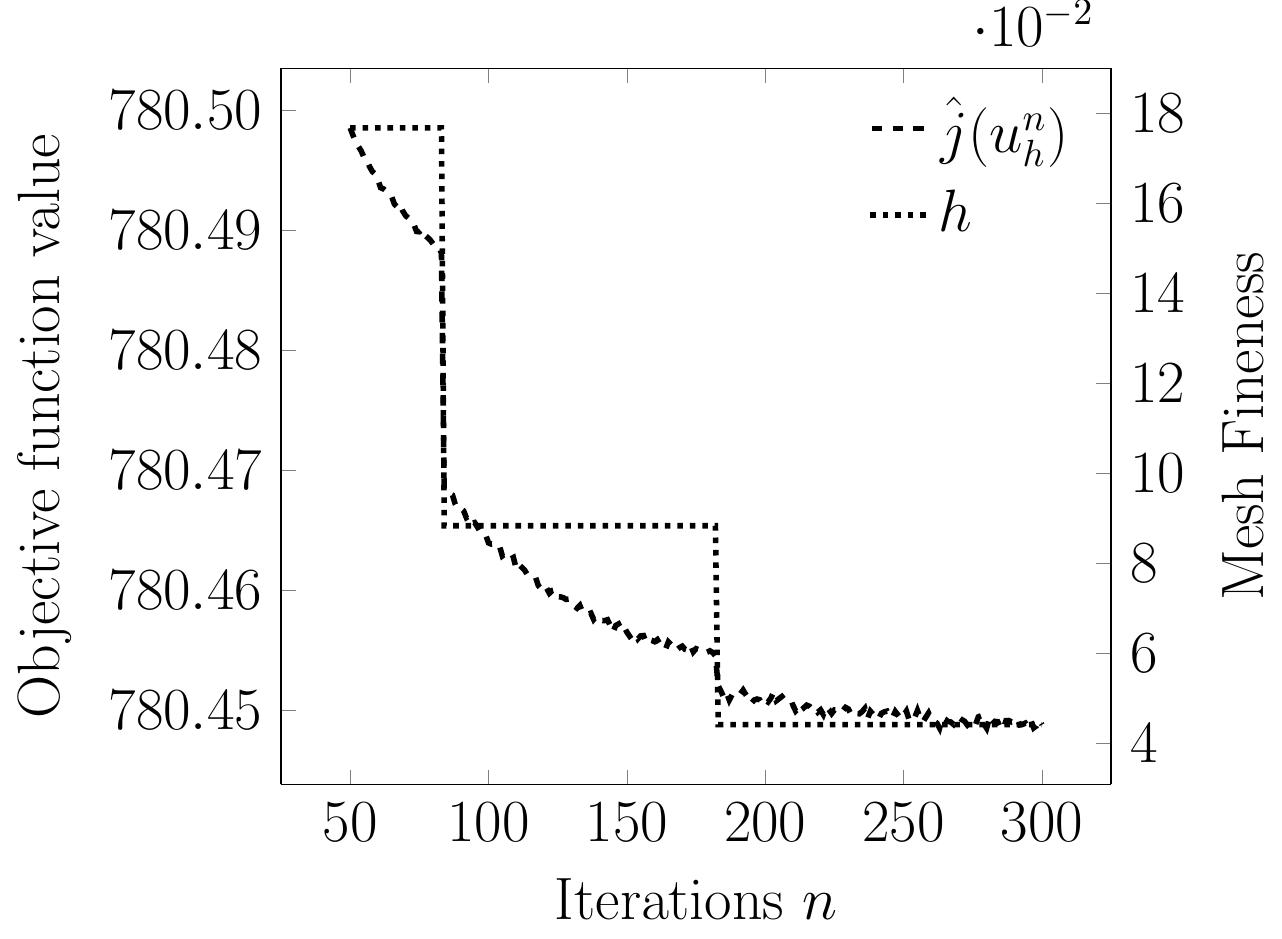}
    \caption{Objective function}  
  \end{subfigure}
   \begin{subfigure}[b]{0.32\linewidth}
  \centering
    \includegraphics[height = 2.8cm]{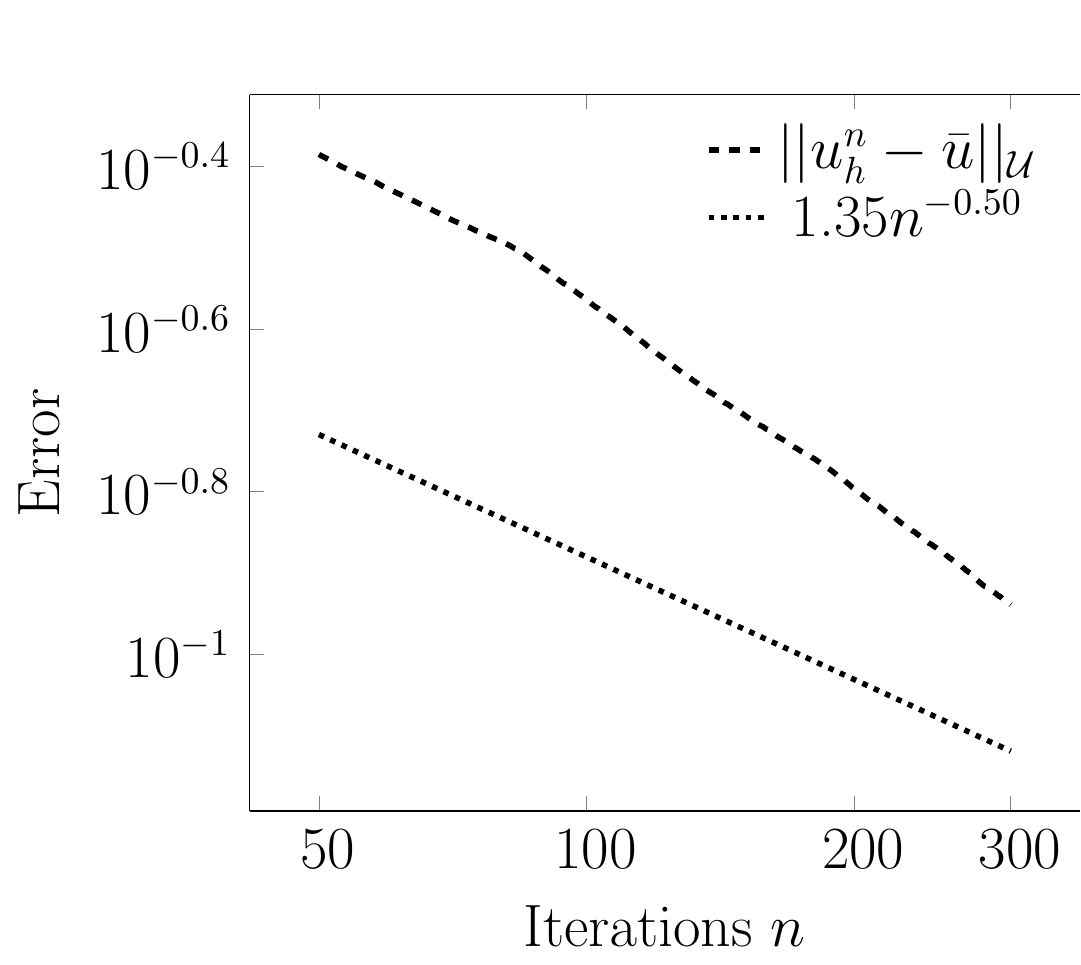}
    \caption{Error in iterates}  
  \end{subfigure}
 \begin{subfigure}[b]{0.32\linewidth}
  \centering
    \includegraphics[height = 2.8cm]{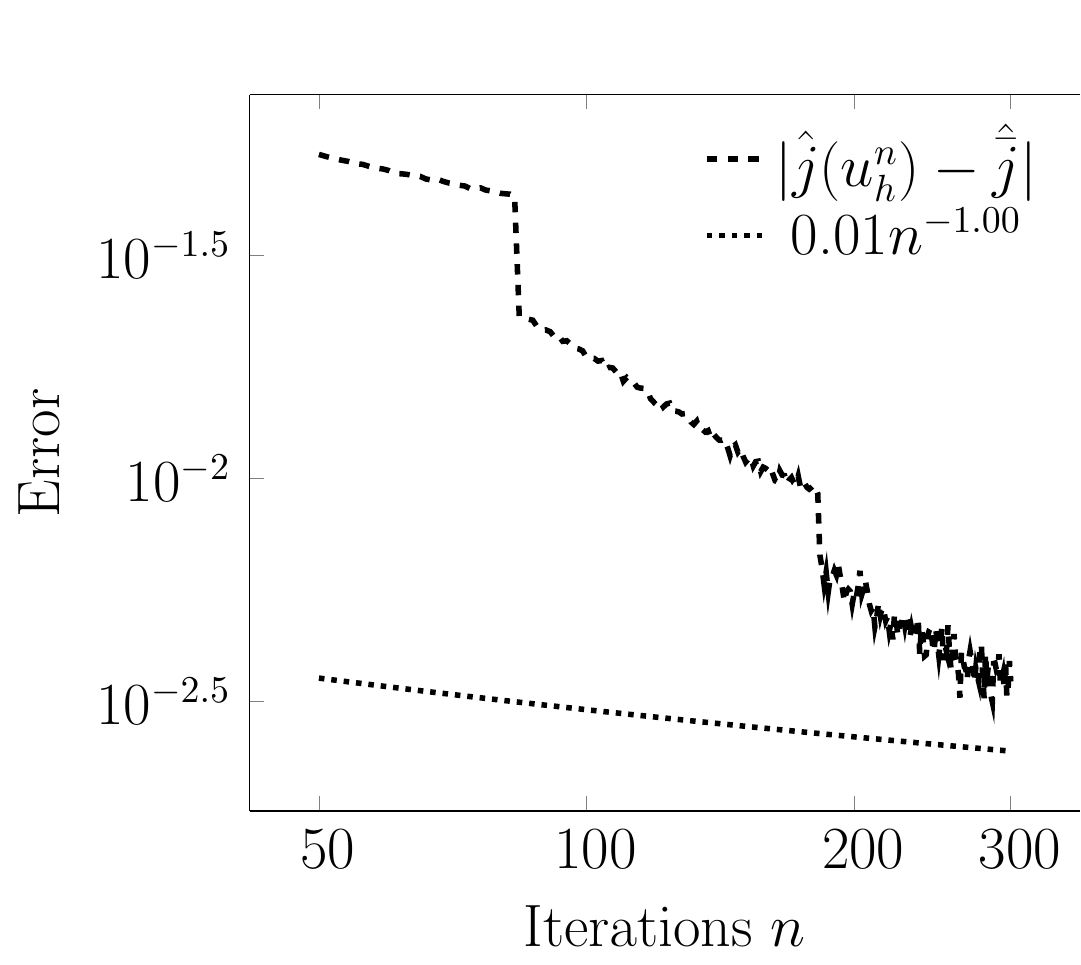}
    \caption{Error in objective function}  
  \end{subfigure}
  \caption{Strongly convex functional with log-normal random field (example 2).}
  \label{fig:SC-RF-2-convergence}
\end{figure}

\begin{figure}
\centering
  \begin{subfigure}[b]{0.32\linewidth}
  \centering
    \includegraphics[height = 2.8cm]{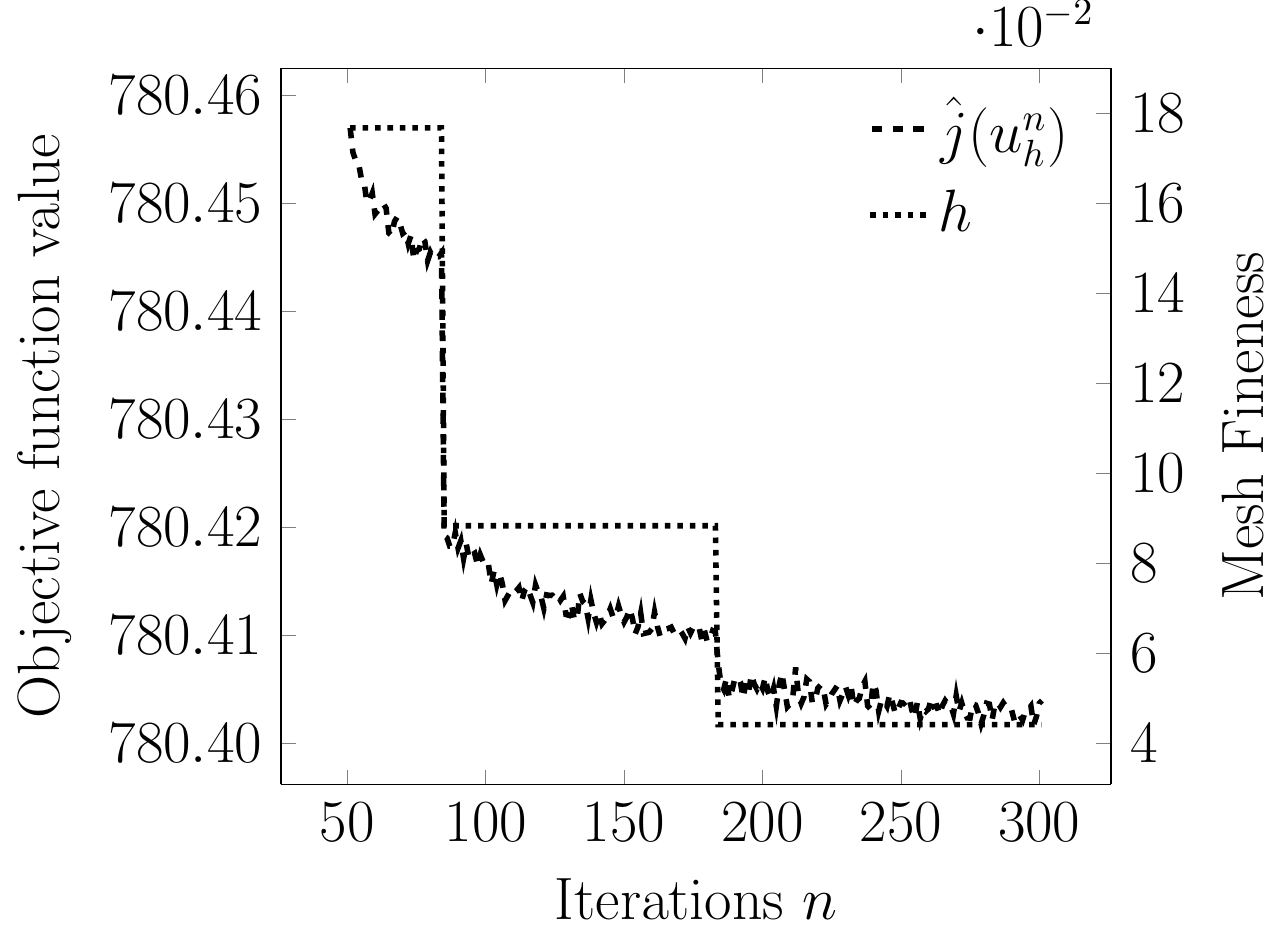}
    \caption{Objective function}  
  \end{subfigure}
   \begin{subfigure}[b]{0.32\linewidth}
  \centering
    \includegraphics[height = 2.8cm]{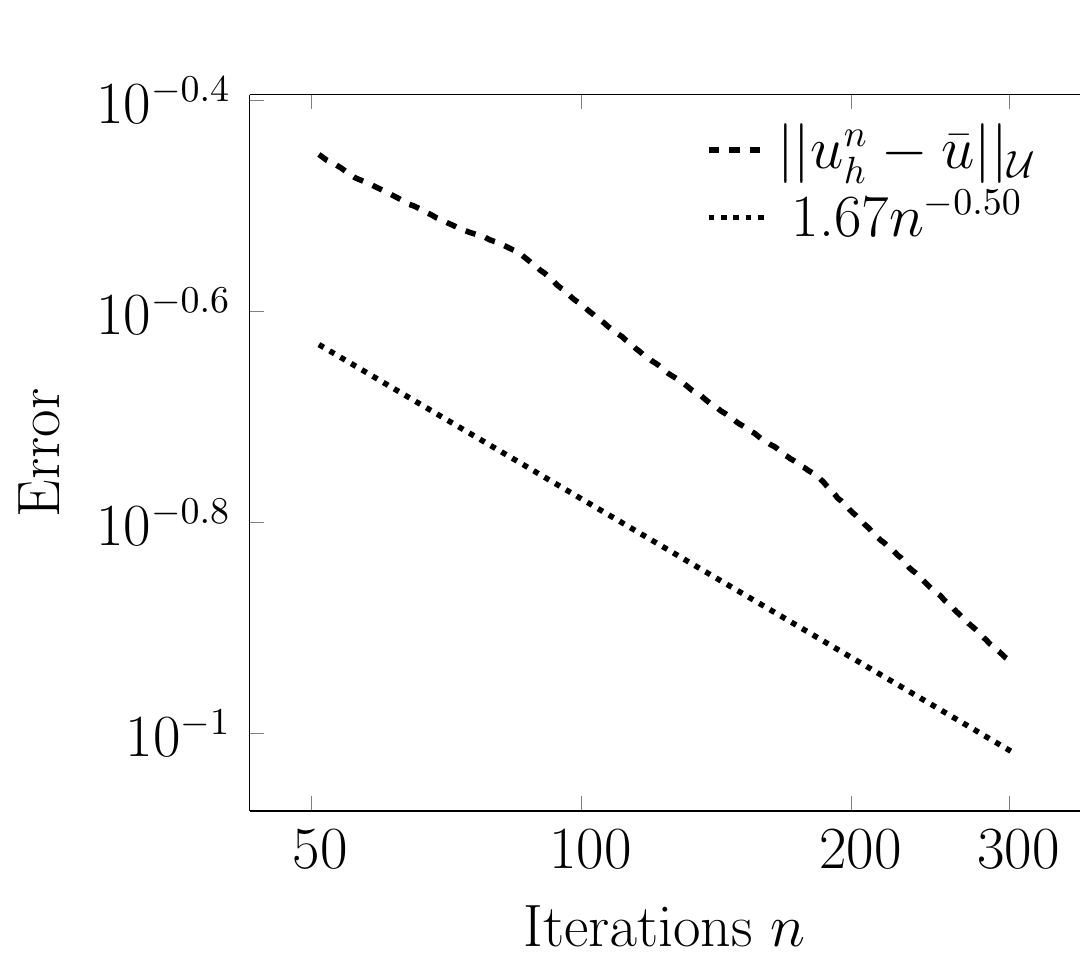}
    \caption{Error in iterates}  
  \end{subfigure}
 \begin{subfigure}[b]{0.32\linewidth}
  \centering
    \includegraphics[height = 2.8cm]{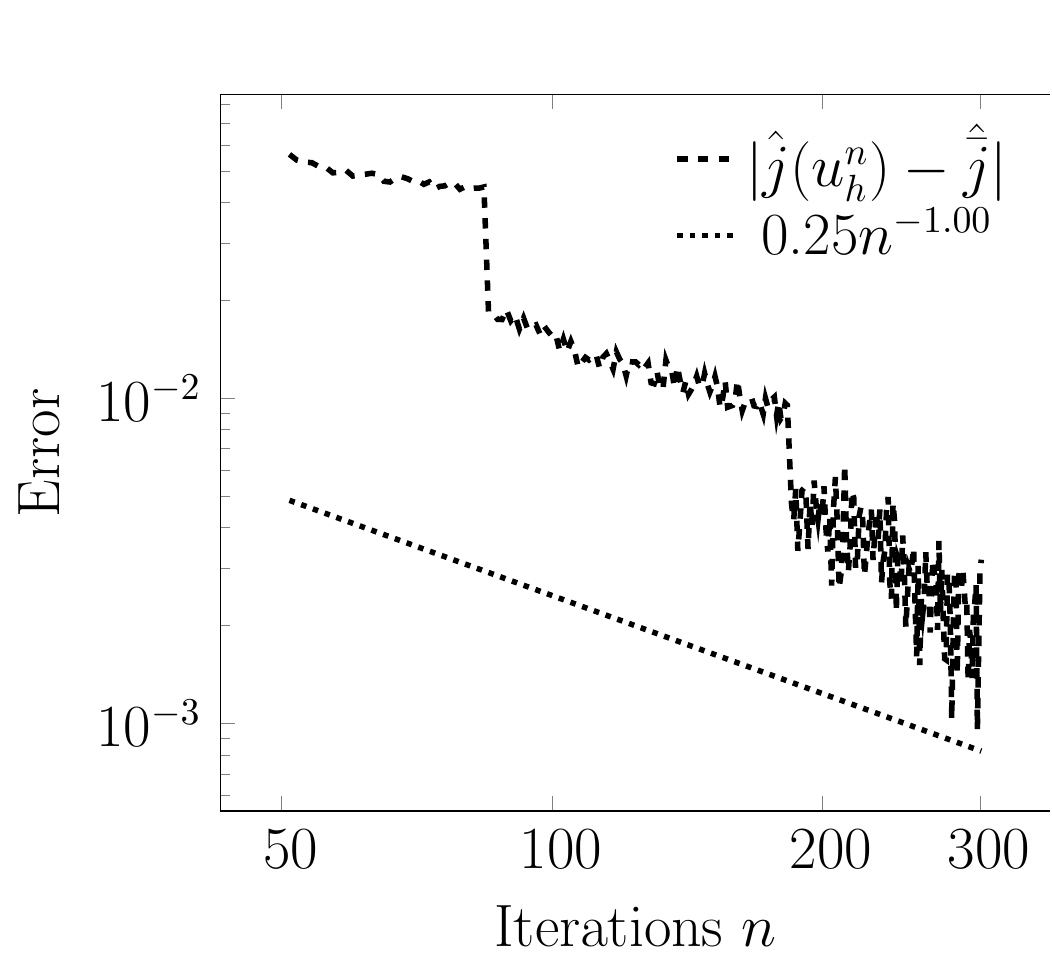}
    \caption{Error in objective function}  
  \end{subfigure}
 
  \caption{Strongly convex functional with piecewise constant random field (example 3).}
  \label{fig:SC-RF-3-convergence}
\end{figure}

\paragraph{Convergence Plots - Averaging}
For the general convex case, we choose the convex example with the modified constraint \eqref{eq:modified-PDE-constraint}. We denote the discretization of the average of iterates $i$ to $N$ $\tilde{u}_{i}^{N}$, defined in  \eqref{eq:averagediterates}, as $\tilde{u}_{i,h}^N.$  We note that the bound on the second moment of the stochastic gradient $M$ can be analytically computed as in \cite{Geiersbach2018} by 
$M=\lVert G(u,\xi) \rVert_{\U}^2 \leq [C(\lVert y^D \rVert_{\U} + C(\lVert u \rVert_{\U} +\lVert e^D\rVert_{\U}))]^2$ with $C = C_p^2/a_{\min}$, where $C_p$ is the Poincar\'e constant, which can be bounded by $\text{diam}(D)/\pi = \sqrt{2}/\pi$ \cite{Payne1960}. Note that $\lVert y^D\rVert_{\U}^2=5/2$, $\lVert e^D\rVert_{\U}^2=1+9\pi^4$ and $\lVert u \rVert_{\U} \leq 1$ for all $u \in \U$.  In addition, for example 1, $a_{\min} \approx 3.55$; for example 2, $a_{\min} \approx 2.72$; for example 3, $a_{\min} = 1$.

To generate reference solutions, the algorithm is run with the variable step size rule \eqref{eq:convex-bias-variable-stepsize-rule} with $\theta = 50$ for $n=5000$ iterations with $h_{5000} \approx 0.0055$ and $\alpha=0.1$ for the averaging factor to get $\bar{u}=\tilde{u}_{4500, h}^{5000}$; see Figure~\ref{fig:reference solutions-C} for the solutions for each random field. To approximate objective function values, $m=5000$ samples were generated to get $\hat{j}(\tilde{u}_{\lceil \alpha N \rceil, h}^N) = \tfrac{1}{m} \sum_{i=1}^m J(\tilde{u}_{\lceil \alpha N \rceil, h}^N,\xi^{n,i})$, where $\xi^{n,i}$ denotes a newly generated $i^{\text{th}}$ sample at iteration $n$. We set $\hat{\bar{j}}:=\hat{j}(\bar{u})$ and use $\alpha = 0.5$ for the experiments. We choose a fixed number of iterations $N \in \{25, 50, \dots, 250\}$ and for each of these iteration numbers, we ran a separate simulation using the step sizes and mesh refinement rules informed by \eqref{eq:convex-bias-constant-stepsize-rule} and \eqref{eq:convex-bias-variable-stepsize-rule}. To prevent the mesh from refining too quickly, we choose $c=2$. For the variable step size rule \eqref{eq:convex-bias-variable-stepsize-rule} we use $\theta=1.$ Plots of convergence for example 1 and example 2 are shown in Figure~\ref{fig:C-RF-1-convergence}--Figure~\ref{fig:C-RF-2-convergence}. Again we see agreement with the theory, with clear jumps when the mesh is refined, both with constant and variable step sizes. We also note that positive jumps in the objective function value are possible when the mesh is refined, as seen in Figure~\ref{fig:C-RF-2-convergence}--Figure~\ref{fig:C-RF-3-convergence}. 
For the third example, we modified the random field so that we can
view the effect of reduced regularity more clearly; we used $\xi \sim
U(5,5.1)$ and $U(1,1.1)$. In
Figure~\ref{fig:C-RF-3-convergence}--Figure~\ref{fig:C-RF-3-convergence},
we see a decrease in convergence rate, which could be caused by missing regularity due to the jump
discontinuity in the random field as mentioned in
Remark~\ref{rem:meshrefinement-lower-regularity}. We reran the experiment
with the guess $\min(2s,t,1)=0.5$, which results in a more aggressive
mesh refinement and convergence according to the theory; see
Figure~\ref{fig:C-RF-3-convergence-c2}. In all examples, the variable step size yields a lower error for the same number of iterations when compared to the constant step size rule.

\begin{figure}
\centering
  \begin{subfigure}[b]{0.32\linewidth}
  \centering
    \includegraphics[height = 3.5cm]{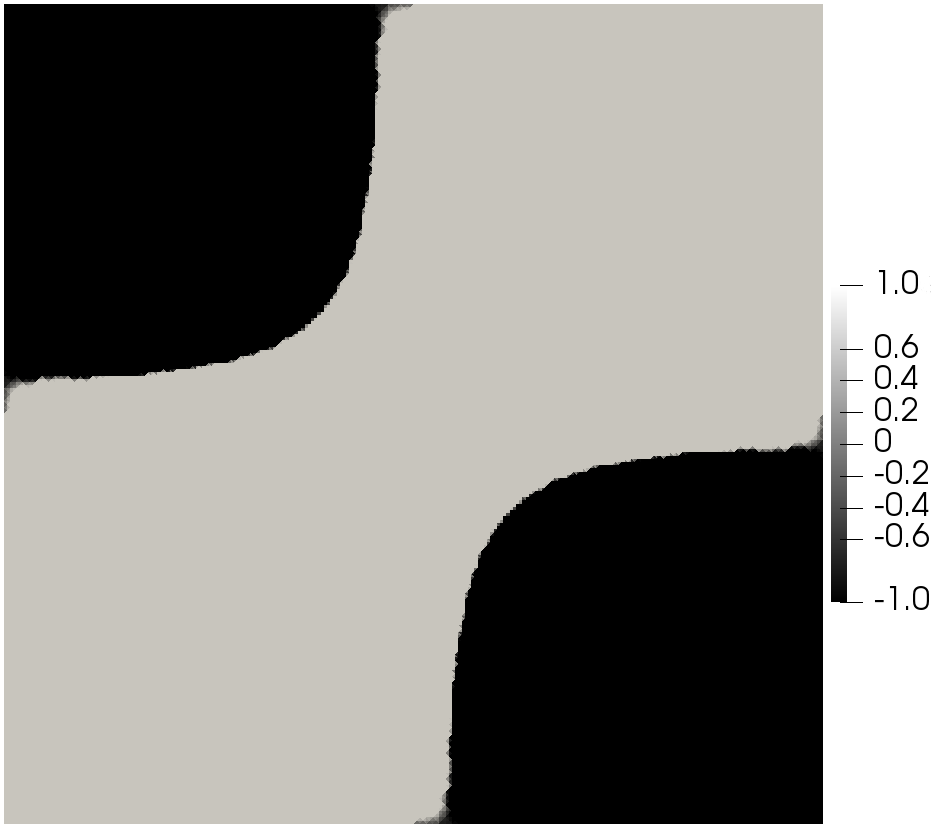}
    \caption{Example 1}  
  \end{subfigure}
  \begin{subfigure}[b]{0.32\linewidth}
  \centering
    \includegraphics[height = 3.5cm]{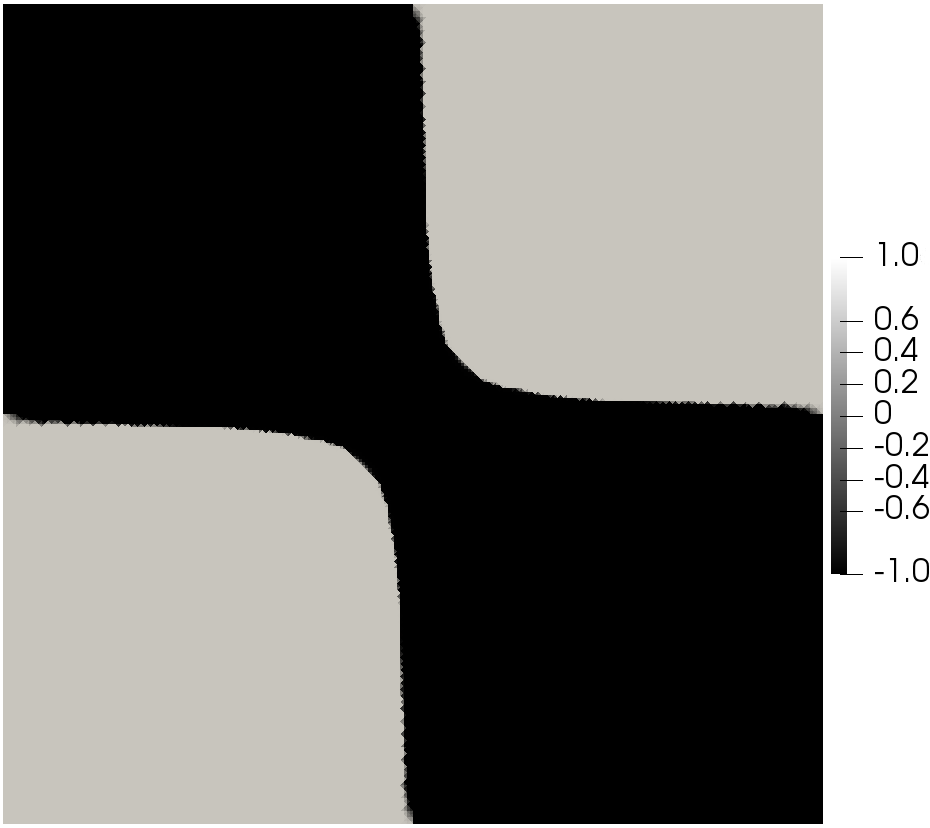}
    \caption{Example 2}  
  \end{subfigure}
   \begin{subfigure}[b]{0.32\linewidth}
  \centering
    \includegraphics[height = 3.5cm]{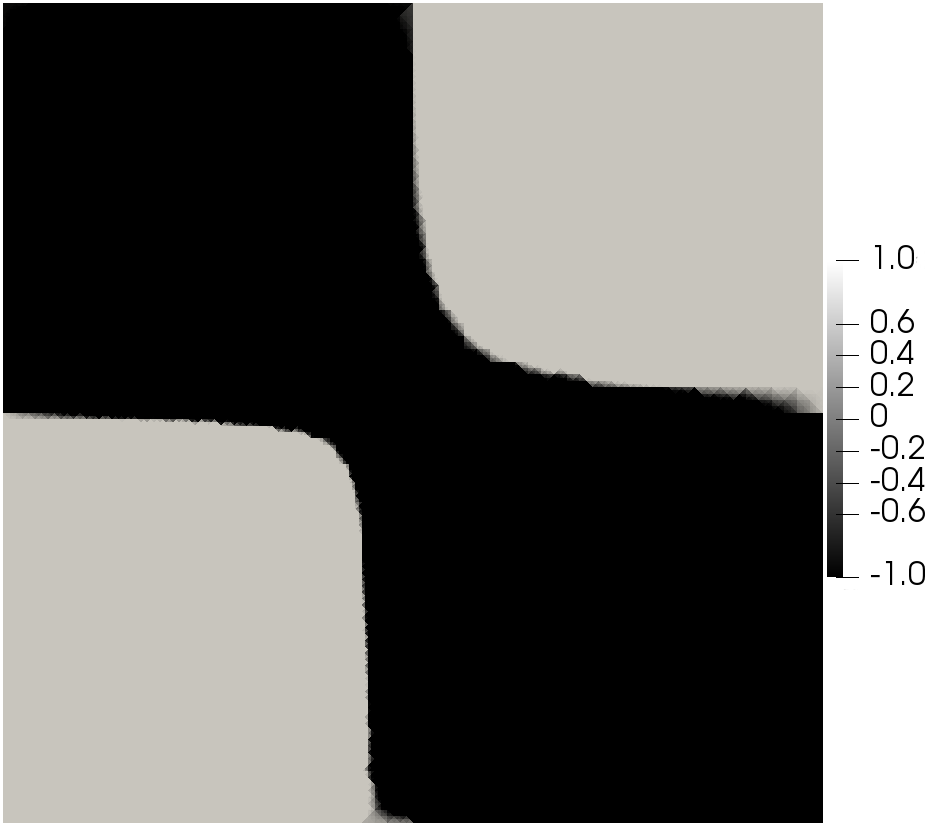}
    \caption{Example 3}  
  \end{subfigure}
  \caption{Reference solutions for general convex case.}
  \label{fig:reference solutions-C}
\end{figure}

\begin{figure}
\centering
  \begin{subfigure}[b]{0.52\linewidth}
  \centering
    \includegraphics[height = 4.0cm]{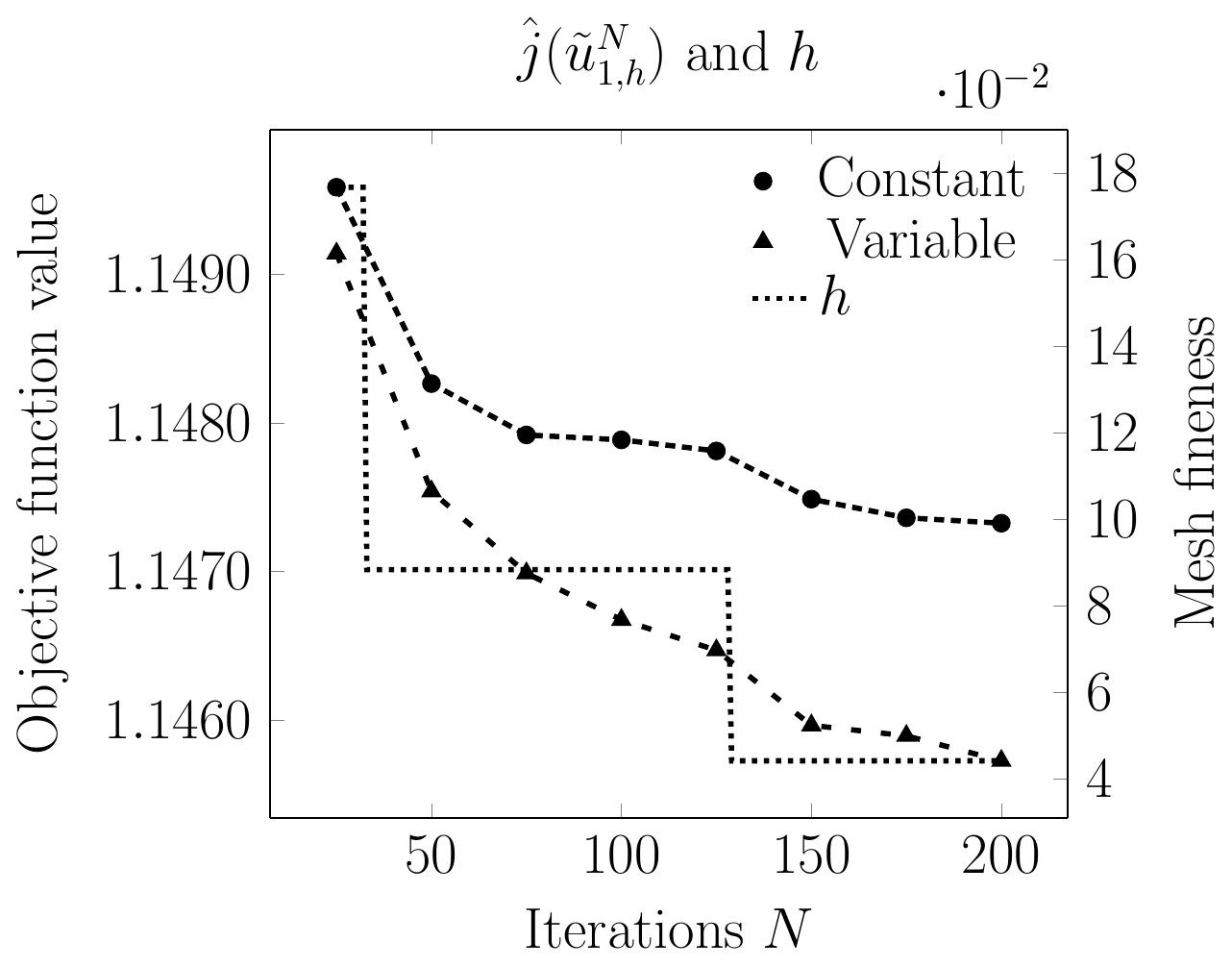}
    \caption{Objective function and mesh fineness}  
  \end{subfigure}
 \begin{subfigure}[b]{0.45\linewidth}
  \centering
    \includegraphics[height = 4.0cm]{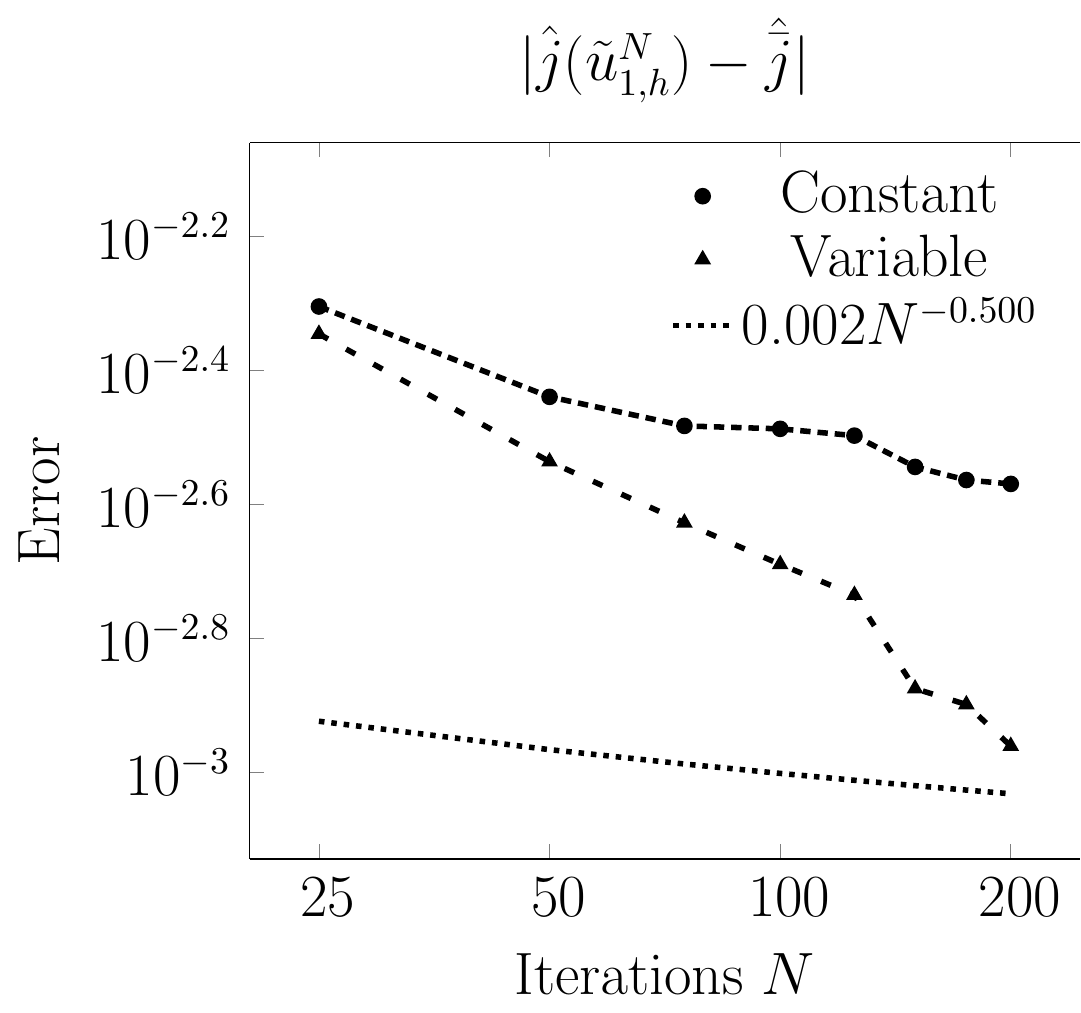}
    \caption{Error in objective function}  
  \end{subfigure}
  \caption{General convex functional with smooth random field (example
    1) using constant and variable step size rules.}
  \label{fig:C-RF-1-convergence}
\end{figure}

\begin{figure}
\centering
  \begin{subfigure}[b]{0.52\linewidth}
  \centering
    \includegraphics[height = 4.0cm]{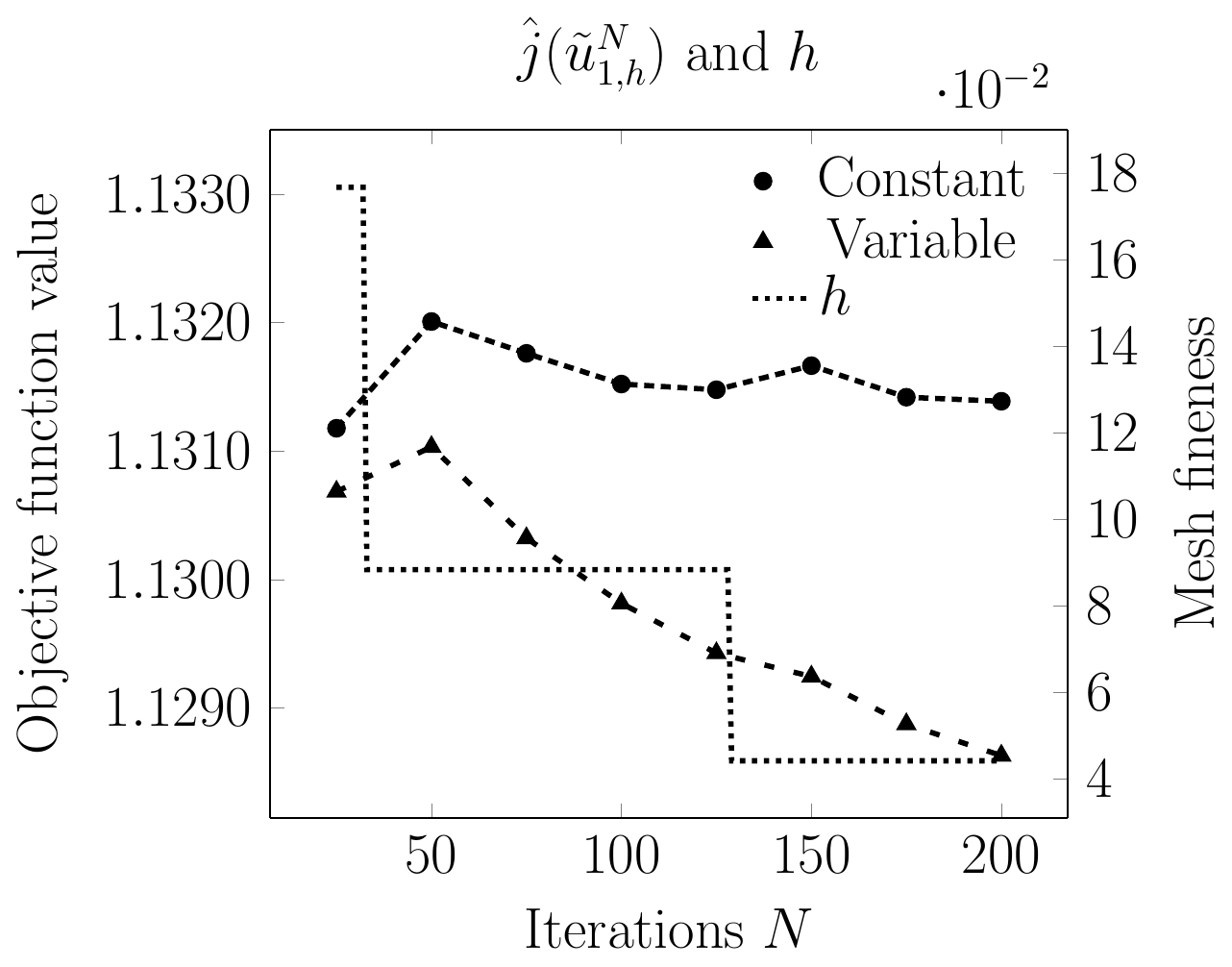}
    \caption{Objective function and mesh fineness}  
  \end{subfigure}
 \begin{subfigure}[b]{0.45\linewidth}
  \centering
    \includegraphics[height = 4.0cm]{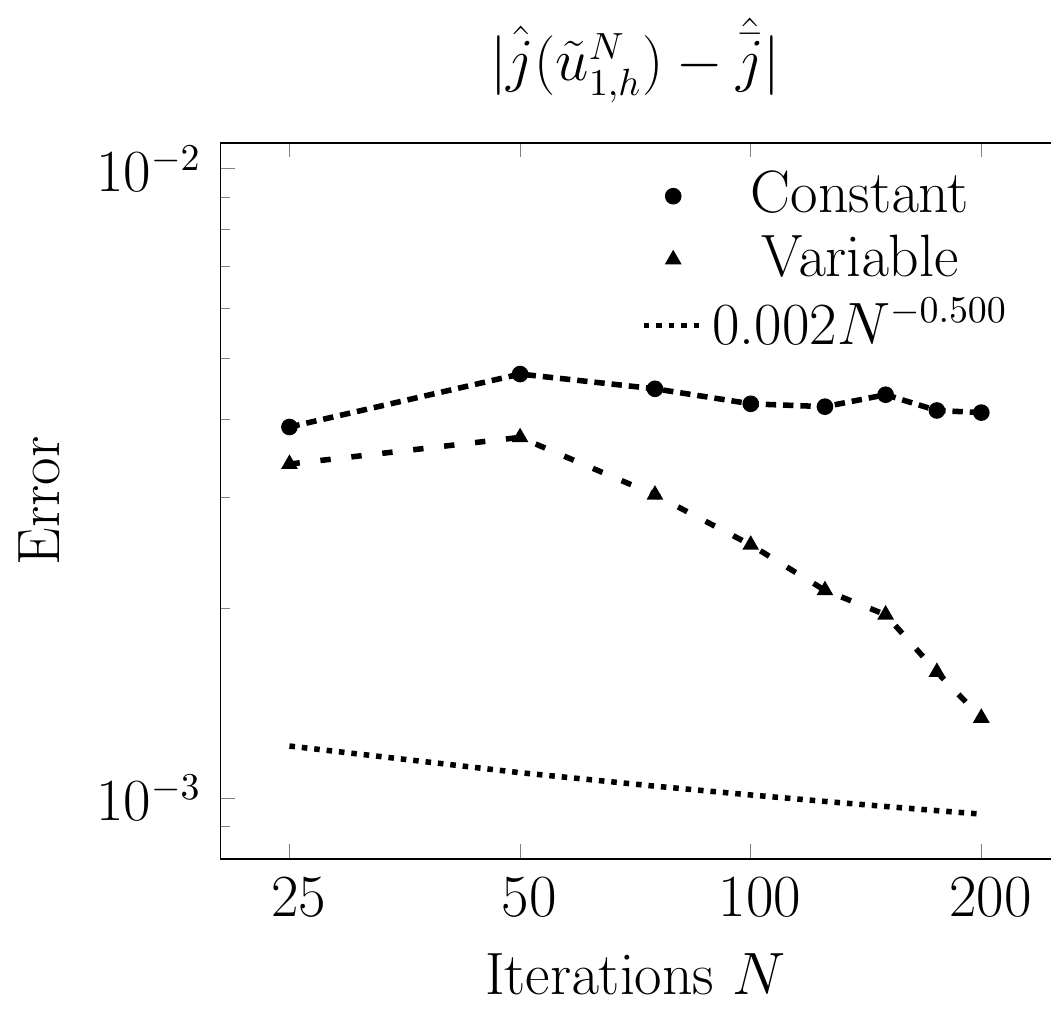}
    \caption{Error in objective function}  
  \end{subfigure}
  \caption{General convex functional with log-normal random field (example 2) using constant and variable step size rules.}
  \label{fig:C-RF-2-convergence}
\end{figure}

\begin{figure}
\centering
  \begin{subfigure}[b]{0.52\linewidth}
  \centering
    \includegraphics[height = 4.0cm]{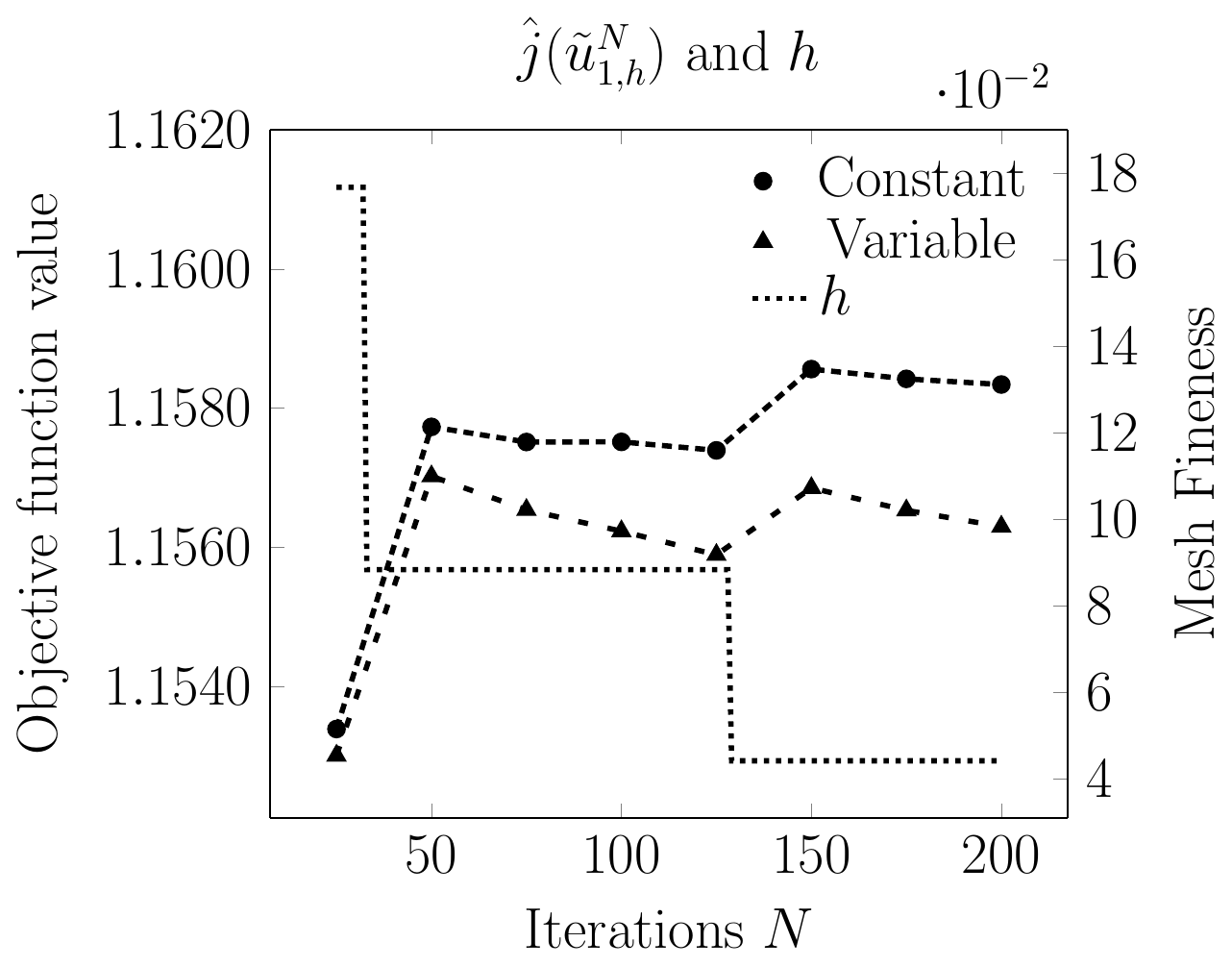}
    \caption{Objective function and mesh fineness}  
  \end{subfigure}
 \begin{subfigure}[b]{0.45\linewidth}
  \centering
    \includegraphics[height = 4.0cm]{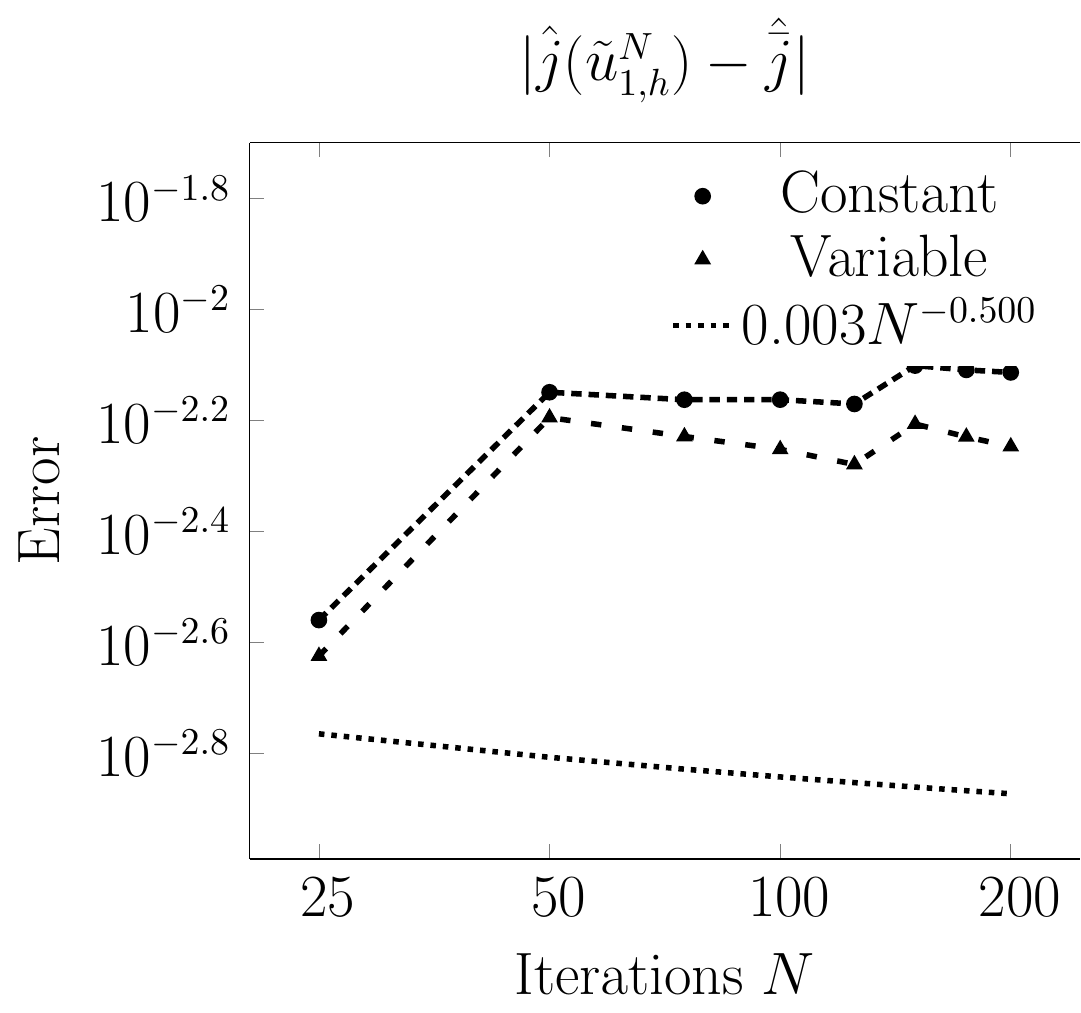}
    \caption{Error in objective function}  
  \end{subfigure}
 
  \caption{General convex functional with piecewise constant random field (example 3) using constant and variable step size rules.}
  \label{fig:C-RF-3-convergence}
\end{figure}

\begin{figure}
\centering
  \begin{subfigure}[b]{0.52\linewidth}
  \centering
    \includegraphics[height = 4.0cm]{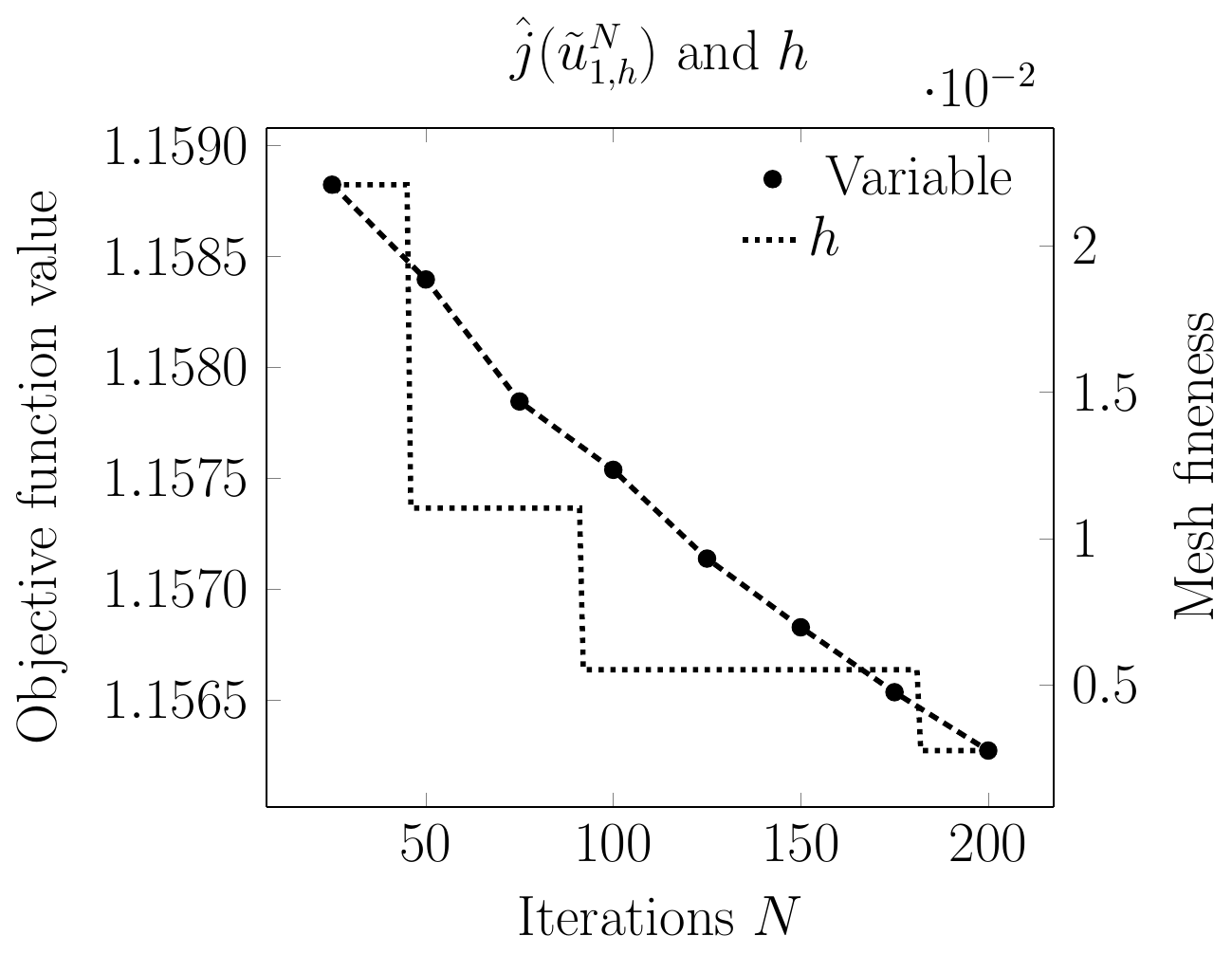}
    \caption{Objective function and mesh fineness}  
  \end{subfigure}
 \begin{subfigure}[b]{0.45\linewidth}
  \centering
    \includegraphics[height = 4.0cm]{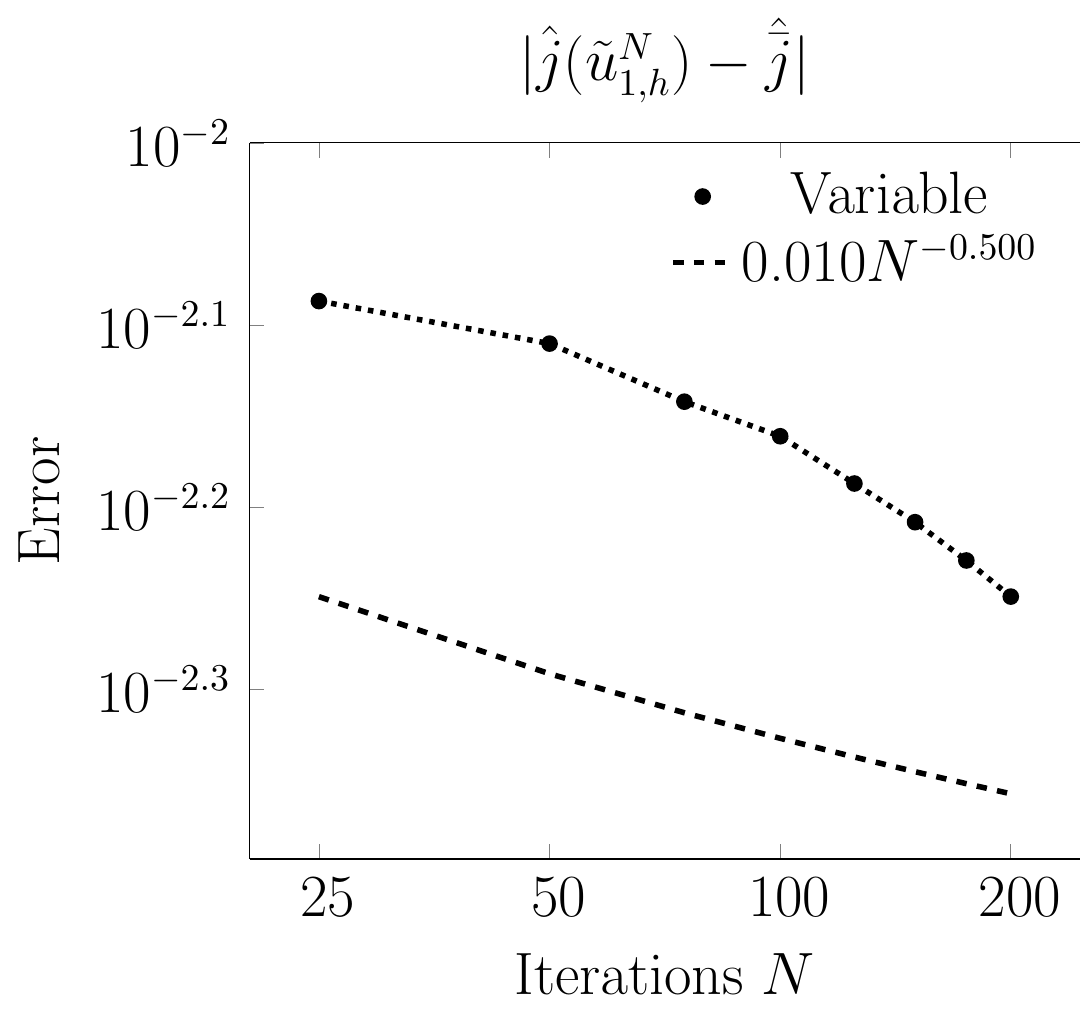}
    \caption{Error in objective function}  
  \end{subfigure}
  \caption{General convex functional with piecewise constant random field (example 3) and $\min(2s,t,1)=0.5$ using variable step size rules.}
  \label{fig:C-RF-3-convergence-c2}
\end{figure}

\section{Conclusion}
\label{section:conclusion}
In this paper, we developed efficiency estimates incorporating numerical error for the projected stochastic gradient algorithm applied to stochastic optimization problems in Hilbert spaces. We distinguish between a strongly convex functional and a general convex case, where in the latter case we use averaging to allow for larger step sizes. These estimates informed how to balance the error and step size rules for both the strongly convex case and the convex case with averaging. We introduced a model stochastic optimization problem with a PDE constraint subject to uncertain coefficients. Using a priori error estimates for the PDE constraint, we developed a mesh refinement strategy that, coupled with reducing step sizes, yields convergence rates according to our efficiency estimates. This was demonstrated using three different random fields on problems with and without a regularization term, which allowed us to test our convergence theory on a strongly convex and general convex objective function.

\bibliographystyle{siamplain}
\bibliography{references}

\begin{thebibliography}{10}

\bibitem{Adams2003}
{\sc R.~A. Adams and J.~J. Fournier}, {\em Sobolev Spaces}, Elsevier, 2003,
  \url{https://doi.org/10.1016/S0079-8169(03)80002-8}.

\bibitem{Alexanderian2017}
{\sc A.~Alexanderian, N.~Petra, G.~Stadler, and O.~Ghattas}, {\em Mean-variance
  risk-averse optimal control of systems governed by {PDE}s with random
  parameter fields using quadratic approximations}, SIAM/ASA J. Uncertain.
  Quantif., 5 (2017), pp.~1166--1192, \url{https://doi.org/10.1137/16m106306x}.

\bibitem{Ali2017}
{\sc A.~A. Ali, E.~Ullmann, and M.~Hinze}, {\em Multilevel {M}onte {C}arlo
  analysis for optimal control of elliptic {PDE}s with random coefficients},
  SIAM/ASA J. Uncertain. Quantif., 5 (2017), pp.~466--492,
  \url{https://doi.org/10.1137/16M109870X}.

\bibitem{Alnes2015}
{\sc M.~S. Aln{\ae}s, J.~Blechta, J.~Hake, A.~Johansson, B.~Kehlet, A.~Logg,
  C.~Richardson, J.~Ring, M.~E. Rognes, and G.~N. Wells}, {\em The {FEniCS}
  project version 1.5}, Arch. of Numerical Software, 3 (2015),
  \url{https://doi.org/10.11588/ans.2015.100.20553}.

\bibitem{Babuska2007}
{\sc I.~Babu{\v{s}}ka, F.~Nobile, and R.~Tempone}, {\em A stochastic
  collocation method for elliptic partial differential equations with random
  input data}, SIAM J. Numer. Anal., 45 (2007), pp.~1005--1034,
  \url{https://doi.org/10.1137/050645142}.

\bibitem{Babuska2004}
{\sc I.~Babu{\v{s}}ka, R.~Tempone, and G.~E. Zouraris}, {\em Galerkin finite
  element approximations of stochastic elliptic partial differential
  equations}, SIAM J. Numer. Anal., 42 (2004), pp.~800--825,
  \url{https://doi.org/10.1137/s0036142902418680}.

\bibitem{Borzi2009}
{\sc A.~Borz\`i and G.~Von~Winckel}, {\em Multigrid methods and sparse-grid
  collocation techniques for parabolic optimal control problems with random
  coefficients}, SIAM J. Sci. Comput., 31 (2009), pp.~2172--2192,
  \url{https://doi.org/10.1137/070711311}.

\bibitem{BrennerScott:2008}
{\sc S.~C. Brenner and L.~R. Scott}, {\em The Mathematical Theory of Finite
  Element Methods}, Springer Verlag, New York, 3.~ed., 2008,
  \url{https://doi.org/10.1007/978-0-387-75934-0}.

\bibitem{Charrier2013}
{\sc J.~Charrier, R.~Scheichl, and A.~L. Teckentrup}, {\em Finite element error
  analysis of elliptic {PDE}s with random coefficients and its application to
  multilevel {M}onte {C}arlo methods}, SIAM J. Numer. Anal., 51 (2013),
  pp.~322--352, \url{https://doi.org/10.1137/110853054}.

\bibitem{Chen2019}
{\sc P.~Chen, U.~Villa, and O.~Ghattas}, {\em Taylor approximation and variance
  reduction for {PDE}-constrained optimal control under uncertainty}, J.
  Comput. Phys., 385 (2019), pp.~163--186,
  \url{https://doi.org/10.1016/j.jcp.2019.01.047}.

\bibitem{Ciarlet:1978}
{\sc P.~G. Ciarlet}, {\em The Finite Element Method for Elliptic Problems},
  vol.~4 of Studies in Mathematics and Applications, North-Holland, 1978,
  \url{https://doi.org/10.1137/1.9780898719208}.

\bibitem{Garreis2019+}
{\sc S.~Garreis and M.~Ulbrich}, {\em A fully adaptive method for the optimal
  control of semilinear elliptic {PDE}s under uncertainty using low-rank
  tensors}, Preprint, Technical University of Munich,  (2019+).

\bibitem{Geiersbach2018}
{\sc C.~Geiersbach and G.~Pflug}, {\em Projected stochastic gradients for
  convex constrained problems in {H}ilbert spaces}, SIAM J. Optim., 29 (2019),
  pp.~2079--2099, \url{https://doi.org/10.1137/18m1200208}.

\bibitem{Grisvard:1985}
{\sc P.~Grisvard}, {\em Elliptic Problems in Nonsmooth Domains}, Monographs and
  studies in Mathematics, Pitman, Boston, 1.~ed., 1985,
  \url{https://doi.org/10.1137/1.9781611972030}.

\bibitem{Gunzburger2014}
{\sc M.~Gunzburger, C.~Webster, and G.~Zhang}, {\em Stochastic finite element
  methods for partial differential equations with random input data}, Acta
  Numer., 23 (2014), pp.~521--650,
  \url{https://doi.org/10.1017/s0962492914000075}.

\bibitem{Guth2019+}
{\sc P.~A. Guth, V.~Kaarnioja, F.~Y. Kuo, C.~Schillings, and I.~H. Sloan}, {\em
  A quasi-{M}onte {C}arlo method for an optimal control problem under
  uncertainty}, arXiv preprint arXiv:1910.10022,  (2019).

\bibitem{Haber2012}
{\sc E.~Haber, M.~Chung, and F.~Herrmann}, {\em An effective method for
  parameter estimation with {PDE} constraints with multiple right-hand sides},
  {SIAM} J. Optim., 22 (2012), \url{https://doi.org/10.1137/11081126x}.

\bibitem{HallerDintelmannMeinlschmidtWollner:2018}
{\sc R.~Haller-Dintelmann, H.~Meinlschmidt, and W.~Wollner}, {\em Higher
  regularity for solutions to elliptic systems in divergence form subject to
  mixed boundary conditions}, Ann. Mat. Pura Appl., 198 (2019), pp.~1227--1241,
  \url{https://doi.org/10.1007/s10231-018-0818-9}.

\bibitem{Hou2011}
{\sc L.~Hou, J.~Lee, and H.~Manouzi}, {\em Finite element approximations of
  stochastic optimal control problems constrained by stochastic elliptic
  {PDE}s}, J. Math. Anal. Appl., 384 (2011), pp.~87--103,
  \url{https://doi.org/10.1016/j.jmaa.2010.07.036}.

\bibitem{Kiefer1952}
{\sc J.~Kiefer and J.~Wolfowitz}, {\em Stochastic estimation of the maximum of
  a regression function}, Ann. Math. Statistics, 23 (1952), pp.~462--466,
  \url{https://doi.org/10.1214/aoms/1177729392}.

\bibitem{Kouri2013}
{\sc D.~Kouri, M.~Heinkenschloss, D.~Ridzal, and B.~G. V.~B. Waanders}, {\em A
  trust-region algorithm with adaptive stochastic collocation for {PDE}
  optimization under uncertainty}, SIAM J. Sci. Comput., 35 (2013),
  pp.~A1847--A1879, \url{https://doi.org/10.1137/120892362}.

\bibitem{Kouri2014}
{\sc D.~Kouri, M.~Heinkenschloss, D.~Ridzal, and B.~V.~B. Waanders}, {\em
  Inexact objective function evaluations in a trust-region algorithm for
  {PDE}-constrained optimization under uncertainty}, SIAM J. Sci. Comput., 36
  (2014), \url{https://doi.org/10.1137/140955665}.

\bibitem{Kroese2011}
{\sc D.~P. Kroese, T.~Taimre, and Z.~I. Botev}, {\em Handbook of {M}onte
  {C}arlo methods}, vol.~706, John Wiley \& Sons, 2011,
  \url{https://doi.org/10.1002/9781118014967}.

\bibitem{Lord2014}
{\sc G.~Lord, C.~Powell, and T.~Shardlow}, {\em An Introduction to
  Computational Stochastic PDEs}, Cambridge University Press, 2014,
  \url{https://doi.org/10.1017/cbo9781139017329.008}.

\bibitem{Lunardi:2018}
{\sc A.~Lunardi}, {\em Interpolation Theory}, Edizioni della Normale, Pisa,
  20018, \url{https://doi.org/10.1007/978-88-7642-638-4}.

\bibitem{Martin2018}
{\sc M.~Martin, S.~Krumscheid, and F.~Nobile}, {\em Analysis of stochastic
  gradient methods for {PDE}-constrained optimal control problems with
  uncertain parameters}, tech. report, {\'E}cole Polytechnique {MATHICSE}
  Institute of Mathematics, 2018.

\bibitem{MeyerRademacherWollner:2015}
{\sc C.~Meyer, A.~Rademacher, and W.~Wollner}, {\em Adaptive optimal control of
  the obstacle problem}, SIAM J. Sci. Comput., 37 (2015), pp.~A918--A945,
  \url{https://doi.org/10.1137/140975863}.

\bibitem{Nemirovski2009}
{\sc A.~Nemirovski, A.~Juditsky, G.~Lan, and A.~Shapiro}, {\em Robust
  stochastic approximation approach to stochastic programming}, SIAM J. Optim.,
  19 (2009), pp.~1574--1609, \url{https://doi.org/10.1137/070704277}.

\bibitem{Payne1960}
{\sc L.~Payne and H.~Weinberger}, {\em An optimal {P}oincar{\'e} inequality for
  convex domains}, Arch. Rational Mech. Anal., 5 (1960), pp.~286--292,
  \url{https://doi.org/10.1007/bf00252910}.

\bibitem{Rannacher2010}
{\sc R.~Rannacher and B.~Vexler}, {\em Adaptive finite element discretization
  in {PDE}-based optimization}, GAMM-Mitt, 33 (2010), pp.~177--193,
  \url{https://doi.org/10.1002/gamm.201010014}.

\bibitem{RannacherVexlerWollner:2011}
{\sc R.~Rannacher, B.~Vexler, and W.~Wollner}, {\em A posteriori error
  estimation in {PDE}-constrained optimization with pointwise inequality
  constraints}, in Constrained Optimization and Optimal Control for Partial
  Differential Equations, vol.~160 of International Series of Numerical
  Mathematics, Springer, 2012, pp.~349--373,
  \url{https://doi.org/10.1007/978-3-0348-0133-1_19}.

\bibitem{Robbins1951}
{\sc H.~Robbins and S.~Monro}, {\em A stochastic approximation method}, Ann.
  Math. Statist., 22 (1951), pp.~400--407,
  \url{https://doi.org/10.1214/aoms/1177729586}.

\bibitem{Rosseel2012}
{\sc E.~Rosseel and G.~Wells}, {\em Optimal control with stochastic {PDE}
  constraints and uncertain controls}, Comput. Methods Appl. Mech. Engrg.,
  (2012), pp.~152--167, \url{https://doi.org/10.1016/j.cma.2011.11.026}.

\bibitem{Schwab2011}
{\sc C.~Schwab and C.~J. Gittelson}, {\em Sparse tensor discretizations of
  high-dimensional parametric and stochastic {PDE}s}, Acta Numer., 20 (2011),
  pp.~291--467, \url{https://doi.org/10.1017/s0962492911000055}.

\bibitem{Tiesler2012}
{\sc H.~Tiesler, R.~M. Kirby, D.~Xiu, and T.~Preusser}, {\em Stochastic
  collocation for optimal control problems with stochastic {PDE} constraints},
  SIAM J. Control Optim., 50 (2012), pp.~2659--2682,
  \url{https://doi.org/10.1137/110835438}.

\bibitem{Triebel:1995}
{\sc H.~Triebel}, {\em Interpolation Theory, Function Spaces, Differential
  Operators}, Johann Ambrosius Barth Verlag; Heidelberg, Leipzig, 2., rev. and
  enl.~ed., 1995.

\bibitem{Ullmann2012}
{\sc E.~Ullmann, H.~C. Elman, and O.~G. Ernst}, {\em Efficient iterative
  solvers for stochastic {G}alerkin discretizations of log-transformed random
  diffusion problems}, SIAM J. Sci. Comput., 34 (2012), pp.~A659--A682,
  \url{https://doi.org/10.1137/110836675}.

\end{thebibliography}
\end{document}